\documentclass[12pt]{article}
\usepackage{amssymb}
\usepackage{amsmath,amscd}
\usepackage{comment}
\usepackage{amsthm}
\usepackage{mathrsfs}
\usepackage[all]{xy}
\usepackage[left=25mm, right=25mm]{geometry}
\usepackage{array}
\usepackage[frozencache,cachedir=.]{minted}

\newcommand{\C}{\mathbb{C}}
\newcommand{\Z}{\mathbb{Z}}
\newcommand{\s}{\sigma}

\newcommand{\OO}{\mathcal{O}}

\usepackage{tikz-feynhand}
\usepackage{tikz}
\usetikzlibrary{arrows,shapes}
\usepackage{tikz-cd}
\usepackage{algorithm}
\usepackage{algpseudocode}
\usepackage{caption}
\newcommand{\DD}{ D\hspace{-.2em}D}

\def\relOmega#1{ \Omega_{X/Z}^#1}

\newtheorem{theorem}{Theorem}[section]
\newtheorem{lemma}[theorem]{Lemma}
\newtheorem{remark}[theorem]{Remark}
\newtheorem{corollary}[theorem]{Corollary}
\newtheorem*{conjecture}{Conjecture}
\newtheorem{definition}[theorem]{Definition}
\newtheorem{proposition}[theorem]{Proposition}
\newtheorem*{assumption}{Assumption}
\newtheorem{example}[theorem]{Example}

\DeclareUnicodeCharacter{03BE}{$\xi$}
\DeclareUnicodeCharacter{03BD}{$\nu$}

\title{Hypergeometric Discriminants}
\author{Saiei-Jaeyeong Matsubara-Heo}
\date{}

\begin{document}

\maketitle

\begin{abstract}
Given a family of varieties, the Euler discriminant locus distinguishes points where  Euler characteristic differs from its generic value.
We introduce a hypergeometric system associated with a flat family of very affine locally complete intersection varieties.
It is proven that the Euler discriminant locus is its singular locus and is purely one-codimensional unless it is empty.
Of particular interest is a family of very affine hypersurfaces.
We coin the term hypergeometric discriminant for the characteristic cycle of the hypergeometric system and establish a formula in terms of likelihood equations.

\end{abstract}

\section{Introduction}

The {\it Euler integral}, an integral of a product of powers of polynomials, provides a way to define hypergeometric functions. This perspective, originating from Aomoto's early work \cite{aomoto1975equations}, led to the study of twisted homology and cohomology \cite{cho1995intersection,goto2023intersection,mimachi2004intersection}. 
Independently, Gelfand, Graev, Kapranov, Retakh, and Zelevinsky introduced the GKZ system \cite{gel1992general,gel1989hypergeometric}, a $D$-module describing a wide class of hypergeometric systems.
A key feature of the GKZ system is its connection to toric ideals \cite{Matusevich-Miller-Walther,saito2013grobner}. 
Solutions to a regular holonomic GKZ system are given by Euler integrals \cite{GKZ}.
This allows for the integration of twisted homology/cohomology studies with the GKZ framework \cite{goto2022homology,matsubara2020euler}.
The study of the singular locus of the GKZ system also led to the development of $A$-discriminants and principal $A$-determinants \cite{GKZbook}. 
Esterov \cite[Theorem 1.7, Theorem 1.16]{esterov2013discriminant} later revisited the principal $A$-determinant  by considering a family of complete intersection varieties $\pi:X=\{ (x,z)\in (\C^*)^n\times Z\mid f_1(x;z)=\cdots=f_\ell(x;z)=0\}\to Z$, where $f_i(x;z)=\sum_{a\in A_i}z_{ia}x^a$ and $Z$ is an Affine space.
When $Z=\prod_{i=1}^\ell \C^{A_i}$, he proved that the vanishing locus of the principal $A$-determinant is precisely where the Euler characteristic of the fiber $\pi^{-1}(z)$ differs from its generic value, defining the {\it Euler discriminant locus}.
Note that this was re-discovered in \cite{amendola2019maximum}.
While Esterov's result extends to a more general parameter space Z under a specific non-degeneracy condition (\cite[Definition 1.18]{esterov2013discriminant}), we do not impose this assumption here.
Our motivation stems from the observation that this non-degeneracy condition is rarely met in physics applications, particularly in the study of Feynman integrals in quantum field theory (\cite{fevola2024landau,fevola2024principal}).

The Euler integral appears as the Lee-Pomeransky representation of a Feynman integral \cite{lee2013critical}, \cite[Section 2.5.4]{weinzierl2022feynman}. 
In this setting, we consider the case where $\ell=1$, and $f(x;z)=f_1(x;z)$ is the graph polynomial associated with a Feynman graph. 
The variables $x$ are the integration variables, known as {\it Schwinger parameters}, while $z=(z_{1a})_{a\in A_1}\in Z\subset\C^{A_1}$ represent the {\it kinematic data}.
Here, $Z$ is typically an affine subspace of $\C^{A_1}$. The Euler integral in this context takes the following form:
\begin{equation}\label{eq:Euler Integral intro}
\int_\Gamma f(x;z)^{-\nu_0}x_1^{\nu_1}\cdots x_n^{\nu_n}\frac{dx_1\cdots dx_n}{x_1\cdots x_n}.
\end{equation}
The study of singularities of a Feynman integral is known as {\it Landau analysis}, and the singular locus is called the {\it Landau variety}. 
Within the framework of Landau analysis, the Feynman integral was principally studied using momentum space representations \cite{landau1960analytic} \cite[Chapter I]{PhamAIHP}.
Recent approaches, however, have increasingly utilized the Euler integral representation \cite{helmer2024landau,mizera2022landau}. 
A geometric way to define the Landau variety is through the concept of the {\it bifurcation locus} of the projection $\pi:X\to Z$, which is the locus where $\pi$ fails to be a fiber bundle.
While classically studied in momentum-space representation, recent approaches also utilize Euler integrals \cite{helmer2024landau,mizera2022landau}.
In fact, computing the bifurcation locus is a long-standing problem in geometry (\cite{JelonekKurdyka,jelonek2017detecting,parusinski1995bifurcation}). 
Instead, one often computes the union of codimension one Whitney strata of a Whitney stratification (\cite[Definition 54]{brown2009periods}, \cite[Definition 1]{helmer2024landau}).
On the other hand, one can also define Landau singularity as the Euler discriminant locus, and the analogy to the principal $A$-determinant is considered in \cite{fevola2024landau,fevola2024principal}.
Note that the Euler discriminant locus is always a subset of the bifurcation locus, but it may in general be a proper subset.
See \cite{joicta2018bifurcation} for a complete description in the case of a polynomial map $\C^{n+1}\to\C^n$.

In the aforementioned literature, Landau singularity is described as either the bifurcation locus, the union of codimension one Whitney strata, or the Euler discriminant locus. However, the Landau singularity should be defined as the locus where the Euler integral \eqref{eq:Euler Integral intro} develops singularities. 
We expect that by introducing a holonomic system to which the Euler integral is subject, its singular locus will provide the correct definition of Landau singularity. Within the framework of the GKZ system (i.e., when $Z=\C^{A_1}$), it has been rigorously proven that the singular locus of the GKZ system is precisely the vanishing locus of the principal $A$-determinant, which is also the Euler discriminant locus. 
Therefore, the first task is to introduce a $D$-module whose singular locus coincides with the Euler discriminant locus. More generally, given a family $\pi:X\to Z$ of locally complete intersection very affine varieties satisfying additional assumptions as detailed in Section \ref{sec:2}, we associate a holonomic $D_Z$-module $M^{\rm hyp}_\pi(\nu)$. 
The following theorem, which will be made more precise in Theorem \ref{thm:ED=Sing2} and Theorem \ref{thm:3.2}, establishes a key connection:
\begin{theorem}
 The singular locus of $M^{\rm hyp}_\pi(\nu)$ is independent of a generic choice of $\nu$ and coincides with the Euler discriminant locus. 
 Furthermore, the Euler discriminant locus is purely one codimensional, meaning that all its irreducible components have codimension one. 
 Moreover, when the family is associated with a single Laurent polynomial $f(x;z)$, this locus is precisely where an Euler integral \eqref{eq:Euler Integral intro} develops singularities for some integration contour $\Gamma$.
\end{theorem}

\noindent
Thus, we argue that the Landau variety should be the Euler discriminant locus. 
An algorithm for computing the Euler discriminant locus has been recently established in \cite{telen2024euler}.

By definition, the {\it Euler discriminant} is an effective divisor \cite[Definition 3.1]{esterov2013discriminant}, which quantifies the drop of the Euler characteristic on each irreducible component of the Euler discriminant locus.
A more refined understanding is provided by the characteristic cycle of $M^{\rm hyp}_\pi(\nu)$, which we term the {\it hypergeometric discriminant}.
The hypergeometric discriminant is also the characteristic cycle corresponding to the constructible function $z\mapsto \chi_z$ (\cite[Section 3]{KennedyCSMclass}).
This is a $(\dim Z)$-dimensional positive algebraic cycle $\sum_{Y}m_YT^*_YZ$ of the cotangent bundle $T^*Z$, where $T^*_YZ$ represents the conormal bundle of a closed subvariety $Y\subset Z$. 
Kashiwara's index theorem (\cite[Theorem~4.3.25]{dimca2004sheaves}, \cite{kashiwara1973index}) implies that the signed Euler characteristic $\chi_z=(-1)^{\dim X_z}\chi(X_z)$ of a fiber $X_z=\pi^{-1}(z)$ is given by the formula $\chi_z=\sum_{Y}(-1)^{{\rm codim} Y}m_YEu_Y(z)$, where $Eu_Y$ is the Euler obstruction of $Y$.
We note that higher codimensional subvarieties $Y\subset Z$ appear in the hypergeometric discriminant.

In principle, the hypergeometric system $M^{\rm hyp}_\pi(\nu)$ arises as a direct image $D$-module of an integrable connection. 
Thus, one can employ general formulas for the characteristic cycles to compute the hypergeometric discriminant such as \cite[Theorem 3.2 and Theorem 9.2]{ginsburg1986characteristic}.
However, a drawback is that one has to choose a relative compactification of the family $\pi:X\to Z$, which is extrinsic data.
In \cite{helmer2025geometric}, the authors provide an algorithm of computing the hypergeometric discriminant employing a particular compactification.
In this paper, we focus more on interpreting the signed Euler characteristic in terms of the {\it maximal likelihood degree} (\cite{catanese2006maximum,huh2013maximum}) to establish an intrinsic formula of the hypergeometric discriminant.
The total space $X$ of the family $\pi:X\to Z$ is a closed subvariety of $T^n\times Z$, where $T^n$ is an $n$-dimensional algebraic torus over $\C$. 
We denote the $i$-th coordinate projection by $x_i:T^n\to\C$. The maximal likelihood degree of a very affine variety is the number of critical points of the log-likelihood function $\sum_{i=1}^n\nu_i\log x_i$ on $X_z$ for generic $\nu_i\in\C$.
When $X_z$ is smooth, the maximal likelihood degree equals the signed Euler characteristic (\cite[Theorem 1.3]{FraneckiKapranov}, \cite[Theorem 1]{huh2013maximum}). 
We denote the fiber product with respect to a given embedding $X\hookrightarrow T^n\times Z$ and the canonical projection ${\rm id}_{T^n}\times\pi_{T^*Z}:T^n\times T^*Z\to T^n\times Z$ by $X\times_{T^n\times Z}(T^n\times T^*Z)$.
To analyze the family and understand the variation of Euler characteristics, it is crucial to consider a family of likelihood equations, which is described by the vanishing locus of the following one-form on $X\times_{T^n\times Z}(T^n\times T^*Z)$:
\begin{equation}\label{eq:the one form}
-\omega_0+\sum_{i=1}^n\nu_id\log x_i.
\end{equation}
Here, $\omega_0$ is the tautological one-form on $T^*Z$, and $d$ is the exterior derivative on $T^n\times T^*Z$. The vanishing locus of the one-form \eqref{eq:the one form} corresponds to an ideal $J_0$ of $\mathcal{O}_{T^n\times T^*Z}$, the structure sheaf of $T^n\times T^*Z$. 
Remarkably, the non-proper direct image sheaf $M_0^{\rm hyp}=pr_*(\mathcal{O}_{T^n\times T^*Z}/J_0)$ is a coherent $\mathcal{O}_{T^*Z}$-module and computes the characteristic cycle. 
Here, $pr:T^n\times T^*Z\to T^*Z$ is the natural projection. 
The natural $\C^*$ action on $T^*Z$ by fiberwise scaling induces a gradation on the structure sheaf $\mathcal{O}_{T^*Z}$, which in turn induces an associated filtration. 
As in the case of $D$-modules, one can associate $M_0^{\rm hyp}$ with its {characteristic cycle} $CC(M_0^{\rm hyp})$ as an algebraic cycle on $T^*Z$, which does not depend on the choice of a good filtration.
By definition, computing $CC(M_0^{\rm hyp})$ does not bypass any non-commutative algebra.
The main result of this paper is the following theorem, established for a family of very affine hypersurfaces $\pi:X\to Z$.
\begin{theorem}\label{thm:theorem 2 intro}
The characteristic cycle $CC(M_0^{\rm hyp})$ is the hypergeometric discriminant.
\end{theorem}
A key technical tool employed to prove the main theorem is the use of $\hbar$-deformation \cite{aoki2004exact,kuwagaki2022hbar,polesello2004stacks}. 
Namely, we introduce a ring $D_{[\![\hbar]\!]}$ of differential operators with a formal parameter $\hbar$ and define an $\hbar$-deformation $M^{\rm hyp}_{[\![\hbar]\!]}$ of the hypergeometric system.
It is rather easily proved that the characteristic cycle of $M^{\rm hyp}_{[\![\hbar]\!]}$ is same as the hypergeometric discriminant.
On the other hand, the transition from the non-commutative module $M^{\rm hyp}_{[\![\hbar]\!]}$ to the commutative one $M^{\rm hyp}_{0}$ is achieved by taking the limit $\hbar\to 0$. 
The main theorem also provides a way to compute the Euler discriminant (locus). 
We also relate hypergeometric discriminants to a description of the Euler discriminant locus recently established by Telen and Wiesmann \cite{telen2024euler} and address some of the questions raised in their paper.
{In fact, computing the Euler stratification is equivalent to computing the hypergeometric discriminants and Euler obstructions.}

The application of the Euler discriminant to the study of Landau singularities is still in its early stages. Although the Euler discriminant locus can be computed using a method from \cite{telen2024euler} or by Theorem \ref{thm:theorem 2 intro}, we currently lack a method to isolate its individual irreducible components.
This likely presents a significant computational drawback. 
The approach used in \cite{fevola2024principal} may offer insights into this challenge. 
Furthermore, the precise relationship between the Euler discriminant locus and the bifurcation locus remains unclear though the latter always contains the former.
Landau analysis is also studied from the viewpoint of microlocal analysis of a holonomic system (e.g., \cite{KashiwaraKawaiFeynman}), where a momentum-space representation is often employed. 
The precise relationship between their holonomic system and our hypergeometric system remains unclear. 
Addressing these open questions will constitute important directions for future research.




\noindent
{\bf Structure of the paper}
In the beginning of \S\ref{sec:2}, we set up notations and assumptions.
In \S\ref{subsec:Mellin}, we introduce a hypergeometric system associated to a family of locally complete intersection very affine varieties $\pi:X\to Z$ via Mellin transform defined in \cite{loeser1991equations}.
In \S\ref{subsec:ED and SL}, we show that Euler discriminant locus is identical with the singular locus of the hypergeometric system under our assumption.
In \S\ref{subsec:Pure codimensionality}, we prove that the Euler discriminant locus is purely one-codimensional.
In \S\ref{subsec:SL and singularity}, we show that the Euler discriminant is the locus where a solution to the hypergeometric system becomes singular.
From \S\ref{sec:Euler integral} on, we focus on a family of very affine hypersurfaces.
In \S\ref{subsec:3.1}, we relate this family with an Euler integral \eqref{eq:Euler Integral intro}.
In \S\ref{subsec:3.2}, we define several rings of differential and difference-differential operators in the presence of a small parameter $\hbar$ to fix notation in the following sections.
The ring $D_{[\![\hbar]\!]}$ is same as the ones discussed in \cite{aoki2004exact,kuwagaki2022hbar,polesello2004stacks} but different from the one discussed in \cite{d2011regular}.
In section \ref{subsec:likelihood formula}, we relate hypergeometric discriminants to likelihood equations (\cite{catanese2006maximum},\cite{huh2013maximum}) when the family $\pi:X\to Z$ is associated to a family of hypersurfaces.
In section \ref{subsec:geometric description}, we compare our description of the Euler discriminant with that of \cite{telen2024euler}.
We use the same notation as the previous section.
Theorem \ref{thm:main result} establishes a new description of the Euler discriminant.
In section \ref{subsec:GKZ case}, we revise the classical GKZ system from our point of view.
The idea of introducing a small parameter $\hbar$ to GKZ system has already appeared in the literature such as \cite{fuji2019reconstructing}.
In section \ref{sec:Landau analysis}, we discuss a family of polynomials $f(x;z)$ known as Lee-Pomeransky polynomial of a Feynman diagram.
Since the Euler discriminant of this family is precisely the singularity of the Euler integral, the section is entitled Landau analysis.
We indicate how to compute hypergeometric discriminants or Euler discriminants for several examples based on Theorem \ref{thm:main result} and conclude the paper with some observations.

\section{Euler discriminant and hypergeometric $D$-module}\label{sec:2}

In the following, $Z$ is a quasi-projective, smooth, and irreducible complex algebraic variety.
For a natural number $n$, we write $T^n$ for the $n$-dimensional algebraic torus ${\rm Spec}\, \C[x_1^{\pm1},\dots,x_n^{\pm1}]$.
We simply write $T$ for $T^n$ when there is no fear of confusion.
For any field extension $\C\subset k$, we write $Z_k$ for the base change of $Z$ by the corresponding morphism ${\rm Spec}\,k\to {\rm Spec}\, \C$ in the sense of scheme theory.
The symbol $K$ denotes the field $\C(\nu)$ of rational functions in $n$ parameters $\nu=(\nu_1,\dots,\nu_n)$.
Let $X$ be a complex algebraic variety and $\pi:X\to Z$ be a morphism of algebraic varieties.
For any $z\in Z$, the symbol $\chi_z$ denotes the absolute value of the Euler characteristic of the fiber $\pi^{-1}(z)$.
By a version of Thom-Mather's first isotopy lemma \cite[Corollaire 5.1]{verdier1976stratifications}, a function $\chi$ defined by $Z\ni z\mapsto \chi_z\in\Z$ is a constructible function.
In particular, $\chi$ takes a constant value on a non-empty Zariski open subset of $Z$.
The symbol $\chi^*$ denotes the value of the function $\chi$ on this open subset .

\begin{definition}\label{def:ED}
    {\it Euler discriminant locus} $\nabla_\chi^\pi(Z)$ of $\pi:X\to Z$ is a closed algebraic subvariety defined as the Zariski closure of the following set:
    \begin{equation}\label{eq:nabla Z}
    {\{ z\in Z\mid \chi_z<\chi^*\}}.    
    \end{equation}
\end{definition}

The definition above is the one that appeared in \cite[Definition 3.2]{fevola2024principal}.
When $\pi:X\to Z$ is a family of very affine varieties, the set \eqref{eq:nabla Z} is Zariski closed by \cite[Theorem 3.1]{telen2024euler}.
The term Euler discriminant first appeared in \cite[Definition 3.1]{esterov2013discriminant} and is, by definition, a Weil divisor, hence of codimension one.
On the other hand, Definition \ref{def:ED} does not imply that the Euler discriminant is always codimension one.

\begin{example}
    Let $\pi:{\rm Spec}\,\C[z_1,z_2]\to\C$ be a morphism defined by $\pi(z_1,z_2):=z_1^2+z_2^2$.
    Then, it is easily seen that $\chi_z=0$ for any $z\neq 0$ and $\chi_0=1$, hence $\nabla^\pi_\chi(Z)=\varnothing$.
\end{example}

In principle, the Euler discriminant locus should be defined as the locus where Euler characteristic drops off.
We impose the following assumptions throughout \S2.

\begin{assumption}
The variety $X$ is a closed, locally complete intersection subvariety of $Z\times T^n$  for some natural number $n$ through a closed embedding $\iota_X:X\to Z\times T^n$.
Any irreducible component of $X$ has a same dimension $\dim X$.
The symbol $\pi :X\to Z$ denotes the composition $\pi_Z\circ\iota_X$ where $\pi_Z:Z\times T^n\to Z$ denotes the natural projection.
The morphism $\pi:X\to Z$ factorizes as $X\overset{\pi}{\to} U\overset{\iota_U}{\to} Z$ where $U\subset Z$ is a non-empty Zariski open subset of $Z$ and $\iota_U$ is an embedding.
We assume that $\pi:X\to U$ is a flat morphism.
\end{assumption}

Under this assumption, it eventually follows that our definition of Euler discriminant locus is compatible with that of \cite[Definition 3.1]{esterov2013discriminant} as we will see in Theorem \ref{thm:pure codimensionality}.
For any $z\in Z$, the symbol $X_z$ denotes the fiber of $\pi$ over $z$.
By our assumption on the morphism $\pi:X\to Z$, the identity
\begin{equation}\label{eq:fiber dimensions}
\dim X_z=\dim X-\dim Z    
\end{equation}
holds for any $z\in Z$ unless $X_z=\varnothing$ and $X_z$ is regarded as a closed locally complete intersection subvariety of $T$.
It follows from \cite[Theorem 4.1]{budur2015signed} that 
\begin{equation}\label{eq:signed e.c.}
    \chi_z=(-1)^{\dim X_z}\chi(X_z)\geq 0.
\end{equation}
In the followings, we introduce a hypergeometric system and relate it to the study of $\nabla^\pi_\chi(Z)$ in this specific setup.
We use standard notation of derived functors for $D$-modules as in \cite{hotta2007d} and those for singularity theory as in \cite{dimca2004sheaves}.

\subsection{Partial Mellin transform and a hypergeometric system}\label{subsec:Mellin}
For a smooth algebraic variety $Y$, we write $D_Y$ for the sheaf of rings of differential operators on $Y$.
Let $R^{\rm pol}$ denotes the non-commutative ring generated by elements $\nu_j,\s_j^{\pm 1}$ $(j=1,\dots,n)$ over $\C$ with relations
\begin{equation}
    [\s_i,\nu_j]=\delta_{ij}\s_j,\ \ [\s_i^{-1},\nu_j]=-\delta_{ij}\s_j,\ \ [\s_i,\s_j]=0,\ \ [\nu_i,\nu_j]=0,
\end{equation}
where $\delta_{ij}$ denotes Kronecker's delta and $[a,b]:=ab-ba$ is the commutator bracket.
Let $\Gamma_T$ be the global section functor on $T$.
The ring $R^{\rm pol}$ is isomorphic to the ring of differential operators $\Gamma_TD_T$ on $T$ via the correspondence $\nu_i\leftrightarrow -x_i\partial_{x_i}$ and $\s_i^{\pm 1}\leftrightarrow x_i^{\pm 1}$.
We write $D_Z\boxtimes R^{\rm pol}$ for the tensor product $D_Z\otimes_{\C}R^{\rm pol}$, which naturally has a structure of a sheaf of rings whose product is induced from those of $D_Z$ and $R^{\rm pol}$ by requiring that any element of $D_Z$ commutes with any element of $R^{\rm pol}$.
The ring isomorphism $R^{\rm pol}\simeq\Gamma_TD_T$ naturally induces another ring isomorphism
$\mathfrak{M}_{Z\times T/Z}:D_Z\boxtimes R^{\rm pol}\to D_{Z\times T}=D_Z\boxtimes D_T$.
When it is obvious from the context, we simply write $\mathfrak{M}$ for $\mathfrak{M}_{Z\times T/Z}$.
Given a $D_{Z\times T}$-module $M$, we define a $D_Z\boxtimes R^{\rm pol}$-module $N=\mathfrak{M}^{\rm pol}(M)$ as follows: $N$ is same as $M$ as a sheaf of abelian groups and the action $\bullet$ of $D_Z\boxtimes R^{\rm pol}$ is given by $P\bullet m:= (\mathfrak{M}P)\cdot m\ \ \ (P\in D_Z\boxtimes R^{\rm pol},\ m\in N)$,
where $\cdot$ is the action of $D_{Z\times T}$ on $M$.
We write $R:=K\otimes_{\C[\nu_1,\dots,\nu_n]}R^{\rm pol}$, which also naturally carries a structure of a non-commutative ring.
We define $\DD := D_Z \boxtimes R$, which naturally carries the structure of a sheaf of $K$-algebras.
The partial Mellin transform of a $D_{Z\times T}$-module is a $\DD$-module $\mathfrak{M}(M):=K\otimes_{\C[\nu_1,\dots,\nu_n]}\mathfrak{M}^{\rm pol}(M)$.
For a $D_{Z\times T}$-module $M$, $M_K:=K\otimes_{\mathbb{C}}M$ naturally carries the structure of a $D_{Z_K\times T_K}$-module.
We set 
$$\relOmega{k}(M):=\pi_{Z*}\left(\bigoplus_{1\leq i_1<\cdots<i_k\leq n}M_Kdx_{i_1}\wedge\cdots\wedge dx_{i_k}\right),$$
where $\pi_Z:Z_K\times T_K\to Z_K$ is the base change of the projection $\pi_Z:Z\times T\to Z$ by abuse of notation.
Note that $\pi_Z$ is an affine morphism.
We consider a relative one-form
$\omega \, =\nu_1\frac{{\rm d}x_1}{x_1} + \cdots +\nu_n\frac{ {\rm d}x_n}{x_n}$.
The \emph{relative twisted differential} $\nabla_\omega: \relOmega{k}(M)\rightarrow \relOmega{{k+1}}(M)$ is defined by $\nabla_\omega(\phi) = ({\rm d}_x + \omega \wedge)\, \phi$.
This gives rise to a cochain complex
\begin{equation}\label{eq:de Rham complex}
(\relOmega{\bullet}(M), \nabla_\omega) : \, 0 \longrightarrow \relOmega{0}(M)\overset{\nabla_\omega}{\longrightarrow} \relOmega{1} (M)\overset{\nabla_\omega}{\longrightarrow} \cdots \overset{\nabla_\omega}{\longrightarrow} \relOmega{n}(M)\longrightarrow 0.
\end{equation}
The top cohomology sheaf of this complex $H^n(\relOmega{\bullet}(M),\nabla_\omega)$ naturally admits a left action of $\DD$.
In fact, the left action of $D_{Z_K}$ on $M_K$ induces that of $H^n(\relOmega{\bullet}(M),\nabla_\omega)$.
On the other hand, a shift operator $\sigma_i$ acts on $[\xi(\nu)]\in H^n(\relOmega{\bullet}(M),\nabla_\omega)$ via $
\sigma_i[\xi(\nu)]=[x_i\xi(\nu+e_i)]$, where $\xi(\nu)$ is a local section of $\relOmega{n}(M)$ and $e_i$ is the $i$-th unit vector.
This action induces that of $\DD$.
The following lemma is a straightforward generalization of \cite[1.2.1 (1)]{loeser1991equations}.

\begin{theorem}\label{thm:relative_TCLI}
Under the notation above, the followings are true:
\begin{itemize}
    \item[(1)]$\mathfrak{M}(M)\simeq H^n(\relOmega{\bullet}(M),\nabla_\omega)$ as $\DD$-modules.
    \item[(2)]$H^j(\relOmega{\bullet}(M),\nabla_\omega)=0$ for any $j\neq n$.
\end{itemize}
    
\end{theorem}


The following lemma follows from an argument analogous to the beginning of \cite[1.3.5]{loeser1991equations}.
\begin{lemma}\label{lem:relative_TCLI}
    For generic $\nu\in\C^{n}$, one has 
    $$
    H^j\int_{\pi_Z}Mx^\nu=0
    $$
    for any $j\neq 0$.
\end{lemma}

For a given $D_{Z_K}$-module $N$ and a vector $a\in\Z^n$, we define a $D_{Z_K}$-module $N^{(a)}$ by twisting its action by a shift.
Namely, $N^{(a)}$ is same as $N$ as a sheaf of abelian groups and we define the action of $P(\nu)\in D_{Z_K}$ on $n\in N$ by $P(\nu)\underset{(a)}{\cdot}n:=P(\nu-a)\cdot n$.
Here, the action of $P(\nu-a)$ on the right-hand side is given by that of $N$.

\begin{proposition}\label{prop:contiguity}
 $\mathfrak{M}(M)^{(a)}$ is isomorphic to $\mathfrak{M}(M)$ as a $D_{Z_K}$-module.
\end{proposition}

\begin{proof}
    We define $\varphi: \mathfrak{M}(M)^{(a)}\to \mathfrak{M}(M)$ by $\varphi\left( [\xi(\nu)]\right):=[x^a\xi(\nu+a)]$.
    It is clear that $\varphi$ respects a scalar multiplication by any $a(z;\nu)\in \OO_{Z_K}$.
    On the other hand, an identity $\varphi(\partial_z\underset{(a)}{\cdot}[\xi(\nu)])=\partial_z\cdot\varphi([\xi(\nu)])$
    is also true by a direct computation.
    Here, $\partial_z$ is any local vector field on $Z_K$.
\end{proof}

Now, let us consider the characteristic cycle $CC(\mathfrak{M}(M))$ of $\mathfrak{M}(M)$.
The following discussions can be found in a standard literature such as \cite[\S 2.1,2.2]{hotta2007d}.
Let $Y$ be a smooth algebraic variety defined over a field of characteristic zero.
We equip $D_{Y}$ with the order filtration so that the associated graded ring ${\rm gr}(D_{Y})$ is isomorphic to $\mathcal{O}_{T^*Y}$, the structure sheaf of the cotangent bundle $T^*Y$.
For a coherent $D_{Y}$-module $N$, the associated graded ${\rm gr}(N)$ with respect to a good filtration of $N$ is a $\mathcal{O}_{T^*Y}$-module .
Let ${\rm Supp}_0({\rm gr}(N))$ denotes the set of minimal associated primes of ${\rm gr}(N)$.
Characteristic cycle $CC(N)$ of $N$ is a positive algebraic cycle
\begin{equation}
    CC(N):=\sum_{\mathfrak{p}\in {\rm Supp}_0({\rm gr}(N))}m_\mathfrak{p} V(\mathfrak{p}),
\end{equation}
which does not depend on the choice of a good filtration on $N$.
Here, we set $m_\mathfrak{p}:={\rm length}_{\mathcal{O}_{T^*Y,\mathfrak{p}}}({\rm gr}(N)_{\mathfrak{p}})$ and $V(\mathfrak{p}):={\rm Spec} \,(\mathcal{O}_{T^*Y}/\mathfrak{p})$ viewed as a closed subscheme of $T^*Y$.
The {\it characteristic variety} ${\rm Char}(N)$ is the support of $CC(N)$.
The {\it singular locus} ${\rm Sing}(N)$ of $N$ is defined by ${\rm Sing}(N)\,=\,\varpi\left({\rm Char}(N)\setminus T^*_{Y} Y\right)$, where $\varpi:T^*Y\to Y$ is the canonical projection and $T^*_{Y} Y$ is the zero section of $T^*Y$.
Note that ${\rm Char}(N)$ is a conic subvariety of $T^* Y$, hence ${\rm Sing}(N)$ is a closed subvariety of $Y$.

\begin{proposition}\label{prop:defined over C}
    There is an algebraic cycle on $T^*Z$ whose base change via $\mathbb{C}\subset K$ is $CC(\mathfrak{M}(M))$. In particular, both ${\rm Char}(\mathfrak{M}(M))$ and ${\rm Sing}(\mathfrak{M}(M))$ are defined over $\C$.
\end{proposition}

\begin{proof}
    Given an ideal $I$ of $\mathcal{O}_{T^*Z_K}$ and a vector $a\in\Z^n$, we set $I^{(a)}:=\{ p(\nu+a)\mid p(\nu)\in I\}$.
    It follows that $(\mathcal{O}_{T^*Z_K}/I)^{(a)}=\mathcal{O}_{T^*Z_K}/I^{(a)}$ as $\mathcal{O}_{T^*Z_K}$-modules, hence an equality $CC(N^{(a)})=\sum_{\mathfrak{p}\in {\rm Supp}_0({\rm gr}(N))}m_\mathfrak{p} V(\mathfrak{p}^{(a)})$.
    By Proposition \ref{prop:contiguity}, we obtain an identity of the characteristic varieties ${\rm Char}(\mathfrak{M}(M))={\rm Char}(\mathfrak{M}(M)^{(a)})$.
    By a standard argument on descent, \cite[Chapter 16, proposition 8]{milne2012algebraic}, the characteristic variety ${\rm Char}(\mathfrak{M}(M))$ is defined over $\mathcal{O}_{T^*Z},$ hence so is each $\mathfrak{p}\in {\rm Supp}_0({\rm gr}(\mathfrak{M}(M)))$.
    Note that an identity $\{ p(\nu)\in K\mid \forall a\in\Z^n, p(\nu+a)=p(\nu)\}$ holds true.
    The argument above proves the proposition.
\end{proof}

We also note that there is a localization $\mathbb{C}[\nu]_{\rm loc}$ of $\C[\nu]$ such that 
an equality $CC(\mathfrak{M}(M))=CC(\mathfrak{M}(M)_{\rm loc}(\nu))$ holds for a generic vector $\nu\in\C^n$.
Here, $\mathfrak{M}(M)_{\rm loc}$ is a version of partial Mellin transform over $\mathbb{C}[\nu]_{\rm loc}$ (\cite[p469]{loeser1991equations}) and $(\nu)$ represents a substitution.
The identity follows from another identity
${\rm gr}(\mathfrak{M}(M)_{\rm loc}(\nu))={\rm gr}(\mathfrak{M}(M)_{\rm loc})(\nu),$
which is deduced from a similar argument as \cite[p469]{loeser1991equations}.
Thus, we obtain the following proposition.

\begin{proposition}
    ${\rm Char}\left( H^0\int_{\pi_Z}Mx^\nu\right)$ is independent of a generic choice of $\nu\in\mathbb{C}^n$.
\end{proposition}

Let us apply the construction above to a local cohomology group.
By our assumption, the local cohomology complex 
$\mathbb{R}\Gamma_X(\mathcal{O}_{Z\times T})$ is purely $({\rm codim}X)$-codimensional .
We write $\mathcal{B}_{X|Z\times T}$ for the $({\rm codim}X)$-th cohomology group of this complex.
It is naturally equipped with a structure of a left $D_{Z\times T}$-module.

\begin{definition}
\begin{enumerate}
    \item     The partial Mellin transform $\mathfrak{M}(\mathcal{B}_{X|Z\times T})$ is called a {\it hypergeometric $D_{Z_K}$-module} and denoted by $M^{\rm hyp}_{\pi}$.
    \item For a generic $\nu\in\mathbb{C}^n$, one sets
$$
M^{\rm hyp}_\pi(\nu):=H^0\int_{\pi_Z}\mathcal{B}_{X|Z\times T}x^\nu
$$
and call it a hypergeometric $D_{Z}$-module.
\end{enumerate}
    
\end{definition}

\subsection{Euler discriminant and the singular locus}\label{subsec:ED and SL}

In this section, we reproduce \cite[Theorem 5.3]{fevola2024euler} in our setup. 
For a given point $z\in Z$, we write $\iota_z:\{ z\}\hookrightarrow  Z$ for the canonical embedding.
We set $
    \chi_z(M)\,=\,\sum_{i=0}^\infty (-1)^i\dim_{\C}{\rm Ext}^i_{D_Z}(M,\mathcal{O}_z^{\rm an})
$, where $\mathcal{O}_z^{\rm an}$ is the ring of convergent series at $z$.
We write $\chi^*$ for the holonomic rank of $M$.
Let $\nabla_\chi( Z,M)$ denotes the closure of a set $\{z\in Z\mid \chi_z(M)<\chi^*\}$.
Note that the set $\{z\in Z\mid \chi_z(M)<\chi^*\}$ may not be a closed subvariety of $ Z$ for a general $M$.
A simple counter example is given by a local cohomology.

\begin{lemma}\label{lem:5.6}
    Let $M$ be a holonomic $\mathcal{D}_Z$-module and let ${\rm Sing}^1(M)$ be the union of the codimension $1$ components of the singular locus of $M$.
    Then, one has an inclusion
    \begin{equation}\label{eqn:codim1 inclusion}
        {\rm Sing}^1(M)\,\subset\, {\nabla_{\chi}( Z,M)}.
    \end{equation}
\end{lemma}

\begin{proof}
Let $z$ be a smooth point of ${\rm Sing}^1(M)$.
The multiplicity of the characteristic cycle at $z$ is greater than or equal to one.
By Kashiwara's index formula \cite{kashiwara1973index} (see also \cite[Theorem~4.3.25]{dimca2004sheaves}), it follows that $\chi_z(M)<\chi^*$, which implies $z\in \nabla_{\chi}( Z,M)$.
The inclusion \eqref{eqn:codim1 inclusion} follows by taking the Zariski closure.

\end{proof}

In the following, we frequently equip a complex algebraic variety $Y$ with the analytic topology.
In the standard literature, the analytic variety obtained in this manner is denoted by $Y^{an}$.
In this paper, we stick to the notation $Y$ for the sake of simplicity of notaion.
We write $\mathcal{L}_\nu$ for the local system on $T$ whose monodromy along the $i$-th coordinate axis is $e^{2\pi\sqrt{-1}\nu_i}$.
Let $\pi_T:Z\times T\to T$ be the natural projection.
By abuse of notation, we write $\mathcal{L}_\nu$ for the local system $\pi_T^{-1}\mathcal{L}_\nu$.
It follows from \cite[Theorem 7.1.1]{hotta2007d}, equations \eqref{eq:fiber dimensions}, \eqref{eq:signed e.c.} and the base change formula that 
$$
\scalebox{0.95}{$
\chi_z(M^{\rm hyp}_\pi(\nu))=\displaystyle\sum_{i=0}^\infty (-1)^i\dim_{\C}
H^i\iota_z^{-1}\mathbb{R}\pi_!\mathcal{L}_{-\nu}[\dim X_z]=\sum_{i=0}^\infty (-1)^i\dim_{\C}
R^i\pi_{z!}\iota_{X_z}^{-1}\mathcal{L}_{-\nu}[\dim X_z]
=\chi_z$},
$$
where $\iota_{X_z}:X_z\to Z\times T$ is a closed immersion and $\pi_z:X_z\to \{z\}$.
Therefore, it follows that $\nabla_\chi(Z,M^{\rm hyp}_\pi(Z))=\nabla_\chi^\pi(Z)$.
Let $\nabla_\chi^\pi(Z)^1$ be the union of the codimension one components of $\nabla_\chi^\pi(Z)$.

\begin{lemma}\label{lem:Euler discriminant and singular locus}
    For generic $\nu\in\mathbb{C}^n$, we have $\nabla_\chi^\pi(Z)^{1}={\rm Sing}^1(M^{\rm hyp}_\pi(\nu)).$
\end{lemma}

\begin{proof}
    It follows from Lemma \ref{lem:5.6} that the right-hand side is contained in the left-hand side.
    The general theory of $D$-module shows that a holonomic $D$-module is an integrable connection outside its singular locus.
    This means that the Euler characteristics does not change on the nonsingular locus,
    which proves an inclusion
    \begin{equation}\label{eq:10}
        \nabla_\chi^\pi( Z)\subset{\rm Sing}(M^{\rm hyp}_\pi(\nu)).
    \end{equation} 
\end{proof}

\subsection{Pure codimensionality of Euler discriminant}\label{subsec:Pure codimensionality}

In this section, we prove the following theorem.

\begin{theorem}\label{thm:pure codimensionality}
Any irreducible component of the Euler discriminant $\nabla_\chi^\pi(Z)$ is of codimension one.
Namely, $\nabla_\chi^\pi(Z)$ is purely one-codimensional unless it is empty.
\end{theorem}

The proof relies on the following lemma.

\begin{lemma}\label{lem:pure codimensionality}
    Let $M$ be a holonomic $\mathcal{D}_Z$-module, let $Z'\subset Z$ be a smooth closed subvariety of codimension greater than or equal to two.
    Assume that $H^0_{[Z']}(M)=H^1_{[Z']}(M)=0$ and $M$ is an integrable connection on $Z\setminus Z'$.
    Then, $M$ is an integrable connection on $Z$.
    In particular, one has ${\rm Sing}(M)\cap Z'=\varnothing$.
\end{lemma}

\begin{proof}
    Let $j:Z\setminus Z'\to Z$ be the natural inclusion.
    Since $Z'$ has codimension greater than or equal to two, $j^*M$ is extended to an integrable connection on $Z$.
    We write $E$ for this extension.
    By the assumption and an exact sequence
    $
    H^0_{[Z']}(M)\to M\to H^0\int_jj^*M\to H^1_{[Z']}(M),
    $
    we obtain an isomorphism $M\simeq H^0\int_jj^*M$.
    On the other hand, it also holds that $H^0_{[Z']}(E)=H^1_{[Z']}(E)=0$ because $Z'$ has codimension greater than or equal to two.
    By the same argument as above, it follows that $E=H^0\int_jj^*E$.
    Thus, $M\simeq H^0\int_jj^*M=H^0\int_jj^*E\simeq E$ is an integrable connection on $Z$.
\end{proof}

\noindent
(Proof of Theorem \ref{thm:pure codimensionality})
To simplify the notation, we set $M=M^{\rm hyp}_\pi(\nu)$ for a generic $\nu\in\mathbb{C}^n$.
It suffices to prove an identity 
${\rm Sing}^1(M)={\rm Sing}(M)$.
In fact, this identity implies that $\nabla_\chi^1(Z)=\nabla_\chi(Z)$ since ${\rm Sing}^1(M)=\nabla_\chi^\pi(Z)^1\subset\nabla^\pi_\chi(Z)\subset {\rm Sing}(M)$ by Lemma \ref{lem:Euler discriminant and singular locus} and \eqref{eq:10}.

Assume that $Z'$ is an irreducible component of ${\rm Sing}(M)$ whose codimension is greater than or equal to two.
We replace $Z'$ with an intersection $V\cap Z'$ for an open subset $V$ of $X$ so that $Z'$ is smooth and do not intersect with any other irreducible component.
In particular, $M$ is an integrable connection on $Z\setminus Z'$.
We prove that
\begin{equation}\label{eq:vanishing of local cohomology}
    H^0_{[Z']}(M)=H^1_{[Z']}(M)=0.
\end{equation}
Let $\pi':X'\to Z'$ be the base change of $\pi:X\to Z$ via $\iota:Z'\to Z$.
For this purpose, it is enough to prove that
\begin{equation}\label{eq:local cohomology isom}
\mathbb{R}\Gamma_{Z'}M\simeq\int_{\iota_{Z'}}M^{\rm hyp}_{\pi'}(\nu)[-{\rm codim}(Z')].
\end{equation}
In fact, \eqref{eq:vanishing of local cohomology} follows from \eqref{eq:local cohomology isom} since $\int_{\iota_{Z'}}$ is an exact functor and ${\rm codim}(Z')\geq 2$.
Lemma \ref{lem:pure codimensionality} implies that $Z'$ is not a part of the singular locus, which is a contradiction.
Let $\pi_{Z'}:Z'\times T\to Z'$ be the natural projection, $X'$ be the fiber product $X\times_Z Z'$ and let $\iota_{X'}:X'\to Z'\times T$ be the composition of the associated morphism $X'\to X$ and $\iota_X$.
We may assume that $X'\neq \varnothing$.
We define a local system $\mathcal{L}'_\nu$ on $Z'\times T$ in the same manner.
We note that $X'$ is again of locally complete intersection and 
$$
{\rm codim}(X):=\dim_{\C} (Z\times T)-\dim_{\C}(X)=\dim_{\C} (Z'\times T)-\dim_{\C}(X')=:{\rm codim}(X')
$$
since $\pi:X\to U$ is a flat family.
With this notation and the Riemann-Hilbert correspondence (\cite[p325, Theorem 14.4]{Borel}, \cite[Theorem 7.2.1]{hotta2007d}), \eqref{eq:local cohomology isom} is equivalent to the following isomorphism:
\begin{equation}\label{eq:topological local cohomology isom}
    \iota_*\iota^!\mathbb{R}\pi_{Z*}\iota_{X*}\iota_X^!\mathcal{L}_{\nu}\simeq
    \mathbb{R}{\iota_{Z'*}}\mathbb{R}\pi_{Z'*}\iota_{X'*}\iota_{X'}^!\mathcal{L}'_{\nu}[-{\rm codim}(Z')].
\end{equation}
To prove \eqref{eq:topological local cohomology isom}, we use the base change formula for the following Catesian diagrams:
$$
\begin{tikzcd}
    Z'\times T\arrow[d,"\pi_{Z'}"']\arrow{r}{\tilde{\iota}_{Z'}} & {Z}\times T\arrow[d, "\pi_{Z}"]\\
    Z' \arrow{r}{\iota} & Z 
\end{tikzcd}
\hspace{2em}
\text{and}
\hspace{2em}
\begin{tikzcd}
    X'\arrow[d,"\iota_{X'X}"']\arrow{r}{\iota_{X'}} & {Z'}\times T\arrow[d, "\tilde{\iota}_{Z'}"]\\
    X \arrow{r}{\iota_X} & Z\times T.
\end{tikzcd}
$$
The desired isomorphism \eqref{eq:topological local cohomology isom} is obtained by a repeated application of base change formulas.
We indicate applications of the base change formula by underlining the relevant terms:
\begin{align*}
\iota_*\underline{\iota^!\mathbb{R}\pi_{Z*}}\iota_{X*}\iota_X^!\mathcal{L}_{\nu}&\simeq \iota_*\mathbb{R}\pi_{Z'*}\underline{\tilde{\iota}_{Z'}^!\iota_{X*}}\iota_X^!\mathcal{L}_{\nu} \\
&\simeq \iota_*\mathbb{R}\pi_{Z'*}\iota_{X'*}\underline{\iota_{X'X}^!\iota_X^!}\mathcal{L}_{\nu} \\
&\simeq \iota_*\mathbb{R}\pi_{Z'*}\iota_{X'*}\iota_{X'}^!\underline{\tilde{\iota}_{Z'}^!\mathcal{L}_{\nu}} \\
&\simeq
    \mathbb{R}{\iota_{Z'*}}\mathbb{R}\pi_{Z'*}\iota_{X'*}\iota_{X'}^!\mathcal{L}'_{\nu}[-{\rm codim}(Z')].
\end{align*}
\qed

The beginning of the proof of Theorem \ref{thm:pure codimensionality} proves the following corollary.

\begin{corollary}\label{cor:ED=Sing}
For a generic $\nu\in\C^n$, one has the following identity:
    $$
    \nabla_\chi^\pi(Z)={\rm Sing}(M^{\rm hyp}_\pi(\nu)).
    $$
\end{corollary}

\subsection{Singular locus and the singularity of a solution}\label{subsec:SL and singularity}

\begin{proposition}\label{prop:singular solution}
    Assume $\dim_{\C}Z=1$.
    Let $M$ be a regular holonomic $D_Z$-module and $0\in {\rm Sing}(M)$.
    Let $\iota_0:\{ 0\}\hookrightarrow Z$ be the canonical inclusion.
    Assume $L^{-1}\iota_0^*M=0$ and the holonomic rank of $M$ is at least one.
    Then, there exists a holomorphic solution of $M$ which is not analytically continued to a neighborhood of the origin.
\end{proposition}

\begin{proof}
    It follows from \cite[Remark~4.16]{kashiwara2003d} and the regularity of $M$ that
    $Hom_{\C}(\iota_0^*M,\C)\,\simeq \,Hom_{D_Z}(M,\widehat{\mathcal{O}_0})\,=\,Hom_{D_Z}(M,\mathcal{O}_0)$.
    Now, the assumption $L^{-1}\iota_0^*M=0$ and the proof of Lemma \ref{lem:5.6} imply that $Hom_{D_Z}(M,\mathcal{O}_0)$ has a smaller dimension than the holonomic rank of $M$.
\end{proof}

Next, we consider the case when $Z$ has a higher dimension.
Let $M$ be a holonomic $D_Z$-module.
Let $\mathcal{O}_{Z}^{an}$ be the sheaf of analytic functions on $Z$.
The solution sheaf $Sol(M):=Hom_{D_Z}(M,\mathcal{O}_{Z}^{an})$ is a constructible sheaf and we write $S(M)$ for the set of points $z$ at which $Sol(M)$ fails to be a local system.
In the notation of \cite[Definition 4.1.11]{dimca2004sheaves}, $S(M)$ is nothing but ${\rm Sing}(Sol(M))$, which is an analytic variety by \cite[Proposition 4.1.12]{dimca2004sheaves}.
Analytically, $S(M)$ is the locus where a solution to $M$ fails to be holomorphic.

\begin{proposition}\label{prop:singular set}
    Let $M$ be a regular holonomic $D_Z$-module.
    Assume $L^{-1}\iota_z^*M=0$ for any $z\in Z$ and the holonomic rank of $M$ is at least one.
    Then, the following identity holds:
    \begin{equation}
        {\rm Sing}^1(M)\,=\,S(M).
    \end{equation}
\end{proposition}
\begin{proof}
    From the classical Hartog's theorem in complex analysis (\cite[Chapter 2, Theorem 5B]{WhitneyAnalyticVarieties}), it follows that $S(M)$ does not have any codimension 2 components. 
    Therefore, we have the inclusion ${\rm Sing}^1(M)\,\supset\,S(M)$. 
    Let $z$ be a generic point of ${\rm Sing}^1(M)$.
    It is enough to prove that $z\in S(M)$ to show the other inclusion ${\rm Sing}^1(M)\,\subset\,S(M)$.
    We take a (germ of a) smooth curve $\iota_\Gamma:\Gamma\hookrightarrow Z$ which passes throgh $z$ and is transverse to ${\rm Sing}(M)$.
    By Cauchy-Kowalewsky-Kashiwara's theorem, it follows that $\iota_\Gamma^*\mathcal{H}om_{D_Z}(M,\mathcal{O}_{Z})=\mathcal{H}om_{D_\Gamma}(\iota_\Gamma^*M,\mathcal{O}_{\Gamma})$ and that $L^j\iota_\Gamma^*M=0$ unless $j=0$.
    Let $i_z:\{ z\}\hookrightarrow \Gamma$ denotes the natural inclusion.
    Then, $L^{-1}i_z^*\iota_\Gamma^*M=L^{-1}\iota_z^{*}M=0$.
    Thus, $\iota_\Gamma^*M$ satisfies the assumption of Proposition \ref{prop:singular solution}.
    Therefore, there exists an element $f\in \mathcal{H}om_{D_\Gamma}(\iota_\Gamma^*M,\mathcal{O}_{\Gamma})=\iota_\Gamma^*\mathcal{H}om_{D_Z}(M,\mathcal{O}_{Z})$ which is singular at $z$.
    This proves that $z\in S(M)$.
\end{proof}

Summing up Theorem \ref{thm:pure codimensionality}, Corollary \ref{cor:ED=Sing} and Proposition \ref{prop:singular set}, we obtain the following theorem.
\begin{theorem}\label{thm:ED=Sing2}
For a generic $\nu\in\C^n$, one has the following identity:
    $$
    \nabla_\chi^\pi(Z)={\rm Sing}(M^{\rm hyp}_\pi(\nu))=S(M^{\rm hyp}_\pi(\nu)).
    $$
\end{theorem}

\section{Euler integral and $\hbar$-deformation}\label{sec:Euler integral}

\subsection{Annihilating ideal of the Euler integral}\label{subsec:3.1}
From this section, we assume that $Z$ is affine.
We consider a non-zero polynomial $f(x;z)\in \mathcal{O}_{Z\times T^n}(Z\times T^n)$.
It has an expansion of the form $f(x;z)=\sum_{\alpha\in A}c_\alpha(z)x^\alpha$, where $c_\alpha(z)\in\mathcal{O}_Z(Z)$.
Let us set $\C^*={\rm Spec}\,\C[y^{\pm 1}]$.
We define the hypersurface $X$ as the vanishing locus of $1-yf(x;z).$
The projection $\pi:X\to Z$ factors as $X\to U\to Z$.
Here, $Z\setminus U$ is $\bigcup_{\alpha\in A}\bigcap_{\beta\neq\alpha}V_Z(c_\beta)$ and $V_Z(c_\beta)$ denotes the vanishing locus of $c_\beta$ on $Z$.
The morphism $X\to U$ is a smooth morphism.
In this construction, we know from \cite[Theorem 3.1]{fevola2024principal} that an identity
$
\nabla_\chi^\pi(Z)=\{ z\in Z\mid \chi_z<\chi^*\}
$
holds.
In the following, we set $T:=T^{n+1}={\rm Spec}\,\C[y^{\pm1},x_1^{\pm1},\dots,x_n^{\pm1}]$.
The algebraic local cohomology group $\mathcal{B}_{X|Z\times T}$ has a presentation $D_{Z\times T}/\mathcal{I}$, where $\mathcal{I}$ is an ideal generated by the following operators:
\begin{equation}
1-yf(x;z),\ \ \ x_j\partial_{x_j} - y\partial_y \, yx_j\frac{\partial f}{\partial x_j}(x;z)\ (j=1,\dots,n),\ \ \  \partial_z-y\partial_yy{\partial_z f}(x;z),
\end{equation}
where $j=1,\dots,n$ and $\partial_z$ is any local vector field on $Z$.
We consider $n+1$ parameters $\nu_0,\dots,\nu_n$.
The hypergeometric $D$-module $M_\pi^{\rm hyp}$ (resp. $M_\pi^{\rm hyp}(\nu)$) associated to this data is denoted by $M^{\rm hyp}$ (resp. $M^{\rm hyp}(\nu)$).
In the ring $\DD$, we can define a left ideal $J$ generated by  the following operators over the field $K$:
\begin{equation}\label{eq:generators of J}
    1-\s_{s}f(\s_\nu;z),\ \ \ \nu_j - \nu_0 \, \s_{\nu_0}\s_{\nu_j}\frac{\partial f}{\partial x_j}(\s_{\nu};z),\ \ \ \partial_z+\nu_0\s_{\nu_0}{\partial_z f}(\s_{\nu};z)
\end{equation}
Here, $\partial_z$ runs over the space of algebraic vector fields on $Z$.

It is convenient to use a de Rham complex representation of $M^{\rm hyp}$ different from \eqref{eq:de Rham complex}.
By abuse of notation, we set
$$
\relOmega{k}=\sum_{I=(i_1,\dots,i_k)}\mathcal{O}_{Z_K\times T_K}\left[\frac{1}{f(x;z)}\right]dx^I,
\ \ \ 
\nabla_\omega={\rm d}_x+\omega\wedge:\relOmega{k}\to\relOmega{{k+1}},
$$
where $\omega=-\nu_0\frac{d_xf}{f}+\sum_{i=1}^n\nu_i\frac{dx_i}{x_i}$.
This gives rise to a cochain complex $(\relOmega{\bullet}, \nabla_\omega) : \, 0 \to\relOmega{0}\overset{\nabla_\omega}{\longrightarrow} \relOmega{1} \overset{\nabla_\omega}{\longrightarrow} \cdots \overset{\nabla_\omega}{\longrightarrow} \relOmega{n}\to 0$ of which the $n$-th cohomology is denoted by $H^n((X/Z)_K,\omega)$.

\begin{proposition}\label{prop:3.1}
    There is an isomorphism $M^{\rm hyp}\simeq H^n((X/Z)_K,\omega)$ of $\DD$-modules such that $[1]\in M^{\rm hyp}$ is sent to a cohomology class represented by $\left[\frac{dx_1}{x_1}\wedge\cdots\wedge \frac{dx_n}{x_n}\right]$. 
\end{proposition}

\begin{proof}
Since $X$ is a smooth hypersurface, we obtain a sequence of isomorphisms:
$$
\mathfrak{M}(\mathcal{B}_{X|Z})\simeq\int_\pi\int_{i_{X_K}}i_{X_K}^\dagger\mathcal{O}_{Z_K\times T_K}y^{\nu_0}x^{\nu}[1]=\int_{\pi\circ i_{X_K}}\mathcal{O}_{X_K}f^{-\nu_0}x^\nu.
$$
Here, we set $x^\nu:=x^{\nu_1}\cdots x^{\nu_n}$.
The last object is represented by a shift of the complex $(\relOmega{\bullet}, \nabla_\omega)$ so that its zero-th cohomology group is $H^n((X/Z)_K,\omega)$.
\end{proof}

For a complex vector $(\nu_0,\dots,\nu_n)$, an integral of the form 
\begin{equation}\label{eq:Euler integral section 3}
\int_\Gamma f(x;z)^{-\nu_0}x^\nu\frac{dx_1}{x_1}\wedge\cdots\wedge \frac{dx_n}{x_n}    
\end{equation}
is called an {\it Euler integral}.
It is a period pairing between the twisted cohomology group and the twisted homology group.
For an elementary account, see \cite{matsubara2023four}.
Proposition \ref{prop:3.1} shows that $M^{\rm hyp}$ represents the $\DD$-module that annihilates the Euler integral.
There is also a cyclic $D_{Z_K}$-submodule ${N^{\rm hyp}}$ of $M^{\rm hyp}$ generated by $\left[\frac{dx_1}{x_1}\wedge\cdots\wedge \frac{dx_n}{x_n}\right]$.
Since ${M^{\rm hyp}}=\sum_{a\in\Z^{n+1}}\left( {N^{\rm hyp}}\right)^{(a)}$, it follows that ${\rm Sing}({N^{\rm hyp}})={\rm Sing}({M^{\rm hyp}})$.
For a fixed generic complex vector $\nu\in\C^{n+1}$, the corresponding $D_Z$-submodule ${N^{\rm hyp}}(\nu)$ of ${M^{\rm hyp}}(\nu)$ is well-defined and the identity ${\rm Sing}({N^{\rm hyp}}(\nu))={\rm Sing}({M^{\rm hyp}}(\nu))$ holds true.
For any $z\in Z$, there is a surjection $I:H_n(X_z;\mathcal{L}_\nu)\ni [\Gamma]\mapsto I_\Gamma:=([\xi]\mapsto \int_\Gamma f^{-\nu_0}x^\nu\xi)\in\mathcal{H}om_{D_Z}({N^{\rm hyp}}(\nu),\mathcal{O}_z)$.
By the definition of ${N^{\rm hyp}}(\nu)$, $I_\Gamma$ is determined by $I_\Gamma\left(\left[\frac{dx_1}{x_1}\wedge\cdots\wedge \frac{dx_n}{x_n}\right]\right)$.
The argument above combined with Theorem \ref{thm:ED=Sing2} proves the following Theorem

\begin{theorem}\label{thm:3.2}
    For a generic vector $\nu\in\C^{n+1}$, the variety ${\rm Sing}(M^{\rm hyp}(\nu))$ is precisely where an Euler integral \eqref{eq:Euler integral section 3} develops singularities for some integration contour $\Gamma$.
\end{theorem}

\begin{example}\label{ex:GKZ}
Let us consider a special case of our construction.
Let $A\subset\Z^n$ be a finite set, the variety $Z$ be $\C^A={\rm Spec}\,\C[z_a;a\in A]$ and let $f$ be of the following form: $f(x;z)=\sum_{a\in A}z_ax^a$.
For any $i=1,\dots,n$, $a_i$ denotes the $i$-th entry of $a\in A$.
The (regular) GG system is a left ideal $J_{\rm GG}$ of $\DD$ generated by
\begin{equation}
\sum_{a\in A}z_a\partial_{z_a}+\nu_0,\ \ \ \sum_{a\in A}a_iz_a\partial_{z_a}+\nu_i,\ \ \ \partial_{z_a}-\s_{\nu_0}\s_\nu^{a},
\end{equation}
which was first introduced in \cite{GGsystem}.
Here, $\s^a$ denotes a product $\prod_{i=1}^n\s_{\nu_i}^{a_i}$.
In fact, let $M_0$ be a one-dimensional $K$-vector space generated by a symbol $e$ which has a structure of a left $\DD$-module given by
\begin{equation}
    \s_{\nu_0}\cdot e:=(-\nu_0)\cdot e,\ \ \  \s_{\nu_i}\cdot e:=e, \ \ \ \partial_{z}\cdot e:=0.
\end{equation}
The symbol $e$ plays the role of a Gamma function $(-1)^{\nu_0}\Gamma(\nu_0)$.
A tensor product $M^{\rm hyp}\otimes_KM_0$ naturally has a structure of a left $\DD$-module by
\begin{equation}
    \s\circ (m\otimes a\cdot e)=\s\cdot m\otimes \s\cdot( a\cdot e),\ \ \partial_z\circ (m\otimes a\cdot e)=(\partial_z\cdot m)\otimes a\cdot e\ \  (m\in M^{\rm hyp},\  a\in K).
\end{equation}
A direct computation proves that there is a unique isomorphism of $\DD$ modules 
    $M^{\rm hyp}\otimes_K M_0\to\DD/J_{\rm GG}$ which sends an equivalence class $1\otimes e(s)$ to an equivalence class represented by $1$.
\end{example}


\subsection{$\hbar$-deformation}\label{subsec:3.2}

Let $Z_{[\![\hbar]\!]}$ be the base change of $Z_K$ via the natural inclusion $K\to K[\![\hbar]\!]$ and let $D_{Z_{[\![\hbar]\!]}}$ be the sheaf of differential operators on $Z_{[\![\hbar]\!]}$.
By definition, $D_{[\![ \hbar]\!]}$ is the subsheaf of $D_{Z_{[\![\hbar]\!]}}$ generated over $K[\![\hbar]\!]$ by $g$ and $\partial^\hbar:=\hbar \partial$ where $g$ is a local section of $\mathcal{O}_{Z_{[\![\hbar]\!]}}$ and $\partial$ is a local vector field on $Z_K$.
The ring $D_{[\![ \hbar]\!]}$ is left and right Noetherian.
Let $(U,g_1,\dots,g_N)$ be an \'etale coordinate.
The dual vector fields to one forms $dg_1,\dots,dg_N$ are denoted by $\partial_1,\dots,\partial_N$.
Then, $U\times \mathbb{A}^N\ni (z,\zeta)\mapsto\sum_{i=1}^N\zeta_idg_i(z)\in T^*U$ defines an isomorphism.
The correspondence $\partial_i^{\hbar}\mapsto \zeta_i$ induces a canonical isomorphism $D_{[\![ \hbar]\!]}/\hbar D_{[\![ \hbar]\!]}\simeq\mathcal{O}_{T^*Z_K}$ of $K$-algebras.
We also set $D_{((\hbar))}:=K((\hbar))\otimes_{K[\![\hbar]\!]}D_{[\![\hbar]\!]}$ which also carries a natural structure of a sheaf of $K$-algebras.
We define a non-commutative ring $\DD_{[\![\hbar]\!]}$ (resp. $\DD_{((\hbar))}$) which is generated by $\mathcal{O}_Z$, elements of the form $\partial_z^\hbar$ with $\partial_z$ a vector field on $Z$, and the difference operators $\s_{\nu_i}$ over $K[\![\hbar]\!]$ (resp. $K((\hbar))$).
Relations in this ring are given by
    \begin{equation} \label{eq:commutators1}
  [\partial_z^\hbar,g]=\hbar\partial_z  g\ \ \ (g\in \mathcal{O}_Z),\ \ 
  [\s_{\nu_i},\s_{\nu_j}]=0 \  \  \text{and}\ \ [\s_{\nu_i},\nu_j]=\delta_{ij}\hbar\s_{\nu_i}.
\end{equation}
The ring $D_{[\![\hbar]\!]}$ (resp. $D_{((\hbar))}$) is naturally a $K[\![\hbar]\!]$-subalgebra (resp. $K((\hbar))$-subalgebra) of $\DD_{[\![\hbar]\!]}$ (resp. $\DD_{((\hbar))}$).
In the following, we consider a field extention $K\hookrightarrow K((\hbar))$ defined by a scaling 
\begin{equation}\label{eq:scaling}
K\ni g(\nu)\mapsto g(\nu/\hbar)\in K((\hbar)).    
\end{equation}
One has a natural isomorphism of $K((\hbar))$-algebras 
\begin{equation}\label{eq:DD scaling}
K((\hbar))\otimes_K\DD\overset{\sim}{\to}\DD_{((\hbar))}    
\end{equation}
specified by $1\otimes z\mapsto z,\ \ 1\otimes \partial_z\mapsto \hbar^{-1}\partial_z^\hbar$, $1\otimes\nu_i\mapsto \nu_i/\hbar$ and $1\otimes\sigma_{\nu_i}\mapsto\sigma_{\nu_i}$.
It induces an isomorphism of $K((\hbar))$-algebras $K((\hbar))\otimes_KD_K\overset{\sim}{\to}D_{K((\hbar))}$ specified by $z\mapsto z,\ \ \partial_z\mapsto \hbar^{-1}\partial_z^\hbar$.
We write $J_{[\![\hbar]\!]}\subset\DD_{[\![\hbar]\!]}$ (resp. $J_{((\hbar))}\subset\DD_{((\hbar))}$) for the left ideal of $\DD_{[\![\hbar]\!]}$ (resp. $\DD_{((\hbar))}$) generated by the operators 
\begin{equation}\label{eq:generators of J_h}
    1-\s_{s}f(\s_\nu;z),\ \ \ \nu_j - s \, \s_{s}\s_{\nu_j}\frac{\partial f}{\partial x_j}(\s_{\nu};z),\ \ \ \partial_z^\hbar+s\s_{s}{\partial_z f}(\s_{\nu};z).
\end{equation}
The $\hbar$-deformed hypergeometric $D$-module $M_{[\![\hbar]\!]}^{\rm hyp}$ is a quotient $\DD_{[\![\hbar]\!]}/J_{[\![\hbar]\!]}$ viewed either as a left $\DD_{[\![\hbar]\!]}$-module or a left $D_{[\![\hbar]\!]}$-module.
We also set $M^{\rm hyp}_0:=(K[\![\hbar]\!]/(\hbar))\otimes_{K[\![\hbar]\!]}M^{\rm hyp}_{[\![\hbar]\!]}$, which is a $\DD_{[\![\hbar]\!]}/\hbar\DD_{[\![\hbar]\!]}\simeq\mathcal{O}_{T^*Z_K\times T_K}$-module.

Let us now define a deformed version of de Rham complex.
This is nothing but 
a push-forward on the level of ${D}_{[\![\hbar]\!]}$-modules (\cite[\S5.3]{kuwagaki2022hbar}).
We set
$$
\Omega^{k}_{(X/Z)[\![\hbar]\!]}=\sum_{I=(i_1,\dots,i_k)}\mathcal{O}_{Z_{[\![\hbar]\!]}\times T}\left[\frac{1}{f(x;z)}\right]dx^I,
\ 
\nabla_\omega^\hbar=\hbar{\rm d}_x+\omega\wedge:\Omega^{k}_{(X/Z)[\![\hbar]\!]}\to\Omega^{k+1}_{(X/Z)[\![\hbar]\!]}.
$$
This gives rise to a complex $(\Omega^{\bullet}_{(X/Z)[\![\hbar]\!]}, \nabla_\omega^\hbar) :  0 \to \Omega^{0}_{(X/Z)[\![\hbar]\!]}\overset{\nabla_\omega}{\longrightarrow} \cdots \overset{\nabla_\omega}{\longrightarrow} \Omega^{n}_{(X/Z)[\![\hbar]\!]}\to 0$.
The $n$-th cohomology of this complex is denoted by $H^n((X/Z)_{K[\![\hbar]\!]},\omega)$.

\begin{lemma}\label{lem:no torsion}
    $H^n((X/Z)_{K[\![\hbar]\!]},\omega)$ has no $\hbar$-torsion.
\end{lemma}

\begin{proof}
    An exact sequence $
        0\to(\Omega^\bullet_{(X/Z)[\![\hbar]\!]},\hbar d+\omega\wedge)\overset{\hbar}{\to}(\Omega^\bullet_{(X/Z)[\![\hbar]\!]},\hbar d+\omega\wedge)\to (\Omega^\bullet_{(X/Z)},\omega\wedge)\to 0$ induces the following exact sequence:
    \begin{equation}\label{eq:27}
        0\to \frac{{\rm Ker}\,(\omega\wedge:\Omega^{n-1}_{X/Z}\to \Omega^{n}_{X/Z})}{\omega\wedge\Omega^{n-2}_{X/Z}}\to H^n((X/Z)_{K[\![\hbar]\!]},\omega)\overset{\hbar}{\to}H^n((X/Z)_{K[\![\hbar]\!]},\omega).
    \end{equation}
    For a vector $(\nu_0,\dots,\nu_n)\in\C^{n+1}$, let the symbol $\omega_\nu$ denotes a one form $-\nu_0\frac{d_xf}{f}+\sum_{i=1}^n\nu_i\frac{dx_i}{x_i}$ on $T^n\setminus V_T(f)$.
    We consider a projection $\pi:V:=\{ (x,z,[\nu])\in X\times\mathbb{P}^{n} \mid \omega_{\nu}(x,z)=0\}\ni (x,z,[\nu])\mapsto [\nu]\in \mathbb{P}^{n}$.
    Note that the smoothness of $X$ implies the smoothness of $V$.
    By a version of Bertini's theorem \cite[Th\'eor\`eme 6.10]{jouanolou42theoremes}, there is a Zariski closed proper subset $Y\subset\mathbb{P}^{n}$ such that each fiber $\pi^{-1}([\nu])$ is smooth for any $[\nu]\in\mathbb{P}^{n}$.
    Moreover, since the dimension $\dim_{\C}V$ is equal to $\dim_\C Z$, it follows that $\omega\wedge$ forms a regular sequence, hence the leftmost term of \eqref{eq:27} is zero, which implies that $H^n((X/Z)_{K[\![\hbar]\!]},\omega)$ has no $\hbar$-torsion.
\end{proof}


\begin{theorem}\label{thm:hbar isomorphism}
    The morphism $M_{[\![\hbar]\!]}^{\rm hyp}\to H^n((X/Z)_{K[\![\hbar]\!]},\omega)$ sending $[1]$ to $\left[\frac{dx_1}{x_1}\wedge\cdots\wedge \frac{dx_n}{x_n}\right]$ is an isomorphism of $\DD_{[\![\hbar]\!]}$-modules.
\end{theorem}

\begin{proof}
The morphism is obviously surjective.
Let us write $N$ for the kernel of the morphism $M_{[\![\hbar]\!]}^{\rm hyp}\to H^n((X/Z)_{K[\![\hbar]\!]},\omega)$.
We obtain an exact sequence $0\to N\to M_{[\![\hbar]\!]}^{\rm hyp}\to H^n((X/Z)_{K[\![\hbar]\!]},\omega)\to 0$.
Taking tensor product with $K[\![\hbar]\!]/(\hbar)$, one obtains an exact sequence $
    0\to N/\hbar N\to M_{[\![\hbar]\!]}^{\rm hyp}\otimes_{K[\![\hbar]\!]} K[\![\hbar]\!]/(\hbar)\to \frac{\Omega^{n}_{X/Z}}{\omega\wedge\Omega^{n-1}_{X/Z}}\to 0$
by Lemma \ref{lem:no torsion}.
The last morphism is an isomorphism.
Therefore, we obtain that $N=\hbar N$.
Inductively, we obtain that $
N=\bigcap_{n\geq 0}\hbar^n N\subset \bigcap_{n\geq 0}\hbar^n(\DD_{[\![\hbar]\!]}/J_{[\![\hbar]\!]})=0$.    
The last equality is a consequence of completion theory.
In fact, it is equivalent to that the quotient $M_{[\![\hbar]\!]}^{\rm hyp}=\DD_{[\![\hbar]\!]}/J_{[\![\hbar]\!]}$ is Haussdorff with respect to the $\hbar$-adic topology.
The last condition can be verified immediately since $1+\hbar P$ is an invertible element of $\DD_{[\![\hbar]\!]}$ for any $P\in\DD_{[\![\hbar]\!]}$ (see e.g., the proof of \cite[Theorem 8.10]{Matsumura}).
\end{proof}

\begin{corollary}\label{cor:no torsion}
    $\DD_{[\![\hbar]\!]}/J_{[\![\hbar]\!]}$ has no $\hbar$-torsion.
\end{corollary}


\section{Hypergeometric discriminants}\label{sec:Hypergeometric Discriminant}

\subsection{Likelihood formula of the hypergeometric discriminants}\label{subsec:likelihood formula}
{\it Hypergeometric discriminant} is the characteristic cycle $E^{\rm hyp}$ of $M^{\rm hyp}$.
The purpose of this section is to study hypergeometric discriminant employing an $\hbar$-deformation.
For any integer $k$, $F_kD_{[\![ \hbar]\!]}$ denotes the $\mathcal{O}_{Z_{[\![ \hbar]\!]}}$-submodule of $D_{[\![ \hbar]\!]}$ generated by $\partial_{i_1}^\hbar\cdots\partial_{i_\ell}^\hbar$ for any $\ell=0,1,\dots,k$ and $1\leq i_1,\dots,i_\ell\leq N$. 
We also define a filtration $F_kD_{((\hbar))} \subset D_{((\hbar))}$ as $F_kD_{((\hbar))} := K((\hbar)) \otimes_K F_kD$. Likewise, we construct a filtration $F_k\mathcal{O}_{T^*Z_K} \subset \mathcal{O}_{T^*Z_K}$ by setting $F_k\mathcal{O}_{T^*Z_K} := (K[[\hbar]]/(\hbar)) \otimes_{K[[\hbar]]} F_kD_{[\![\hbar]\!]}$. Thus, $D_{[\![ \hbar]\!]}, D_{((\hbar))}$, and $\mathcal{O}_{T^*Z_K}$ are filtered rings.
For a coherent $R$-module $M$ where $R$ is either $D_{[\![ \hbar]\!]},D_{((\hbar))}$ or $\mathcal{O}_{T^*Z_K}$, a filtration $\{ F_kM\}_{k\in\Z}$ of $M$ is a good filtration if the associated graded ${\rm gr}(M)$ is a finitely generated ${\rm gr}(R)$-module.
The associated algebraic cycle is a formal sum
$CC(M):=\sum_{\mathfrak{p}\in {\rm Supp}_0({\rm gr}(M))}m_{\mathfrak{p}}V(\mathfrak{p})$, which we call the characteristic cycle.
The characteristic cycle $CC(M)$ does not depend on the choice of a good filtration of $M$.
In particular, the characteristic cycle of an $\mathcal{O}_{T^*Z_K}$-module is an algebraic cycle on $T^*Z_K$, as there is a canonical isomorphism $\mathcal{O}_{T^*Z_K}\simeq {\rm gr}(\mathcal{O}_{T^*Z_K})$.
The following theorem is the main result of this section.

\begin{theorem}\label{thm:main result}
\begin{equation*}
    E^{\rm hyp}=CC(M^{\rm hyp}_0).
\end{equation*}    
\end{theorem}
In the rest of this section, we prove Theorem \ref{thm:main result}.

\begin{theorem}\label{thm:3.7}
    $M_{[\![\hbar]\!]}^{\rm hyp}$ is a coherent left $D_{[\![\hbar]\!]}$-module.
\end{theorem}

\begin{proof}
    In view of the fact that $M_{[\![\hbar]\!]}^{\rm hyp}=H^n((X/Z)_{K[\![\hbar]\!]},\omega)$ is $\hbar$-adically Hausdorff by Theorem \ref{thm:hbar isomorphism}, it is enough to show that $(K[\![\hbar]\!]/(\hbar))\otimes_{K[\![\hbar]\!]} H^n((X/Z)_{K[\![\hbar]\!]},\omega)=\frac{\Omega^{n}_{(X/Z)}(X)}{\omega\wedge\Omega^{n-1}_{(X/Z)}(X)}$ is a coherent $\mathcal{O}_{T^*Z_K}$-module.
    Let $k$ be the algebraic closure of $K$.
    This is because $H^n((X/Z)_{K[\![\hbar]\!]},\omega)$ is $\hbar$-adically complete.
    We set $\Omega_{X_k/Z_k}^\ell:=k\otimes_K\Omega_{X/Z}^\ell$.
    By extending the coefficients, it is enough to prove that $\frac{\Omega^{n}_{X_k/Z_k}}{\omega\wedge\Omega^{n-1}_{X_k/Z_k}}$ is a coherent $\mathcal{O}_{T^*Z_k}$-module.
    We choose a system of \'etale coordinates $(V;g_1,\dots,g_N)$.
    Namely, $V\subset Z_k$ is an open subset, $g_1,\dots,g_N\in\mathcal{O}_{Z_k}(V)$, and $(g_1,\dots,g_N):V\to {\rm Spec}(k[t_1,\dots,t_N])$ is an \'etale morphism.
    A system of \'etale coordinates $(V;g_1,\dots,g_N)$ induces an isomorphism $V\times k^N\ni (z,\zeta)\mapsto \sum_{j=1}^N \zeta_jdg_j(z)\in T^*V\subset T^*Z_k$.
    Let $\omega_{0}$ be the tautological one-form on $T^*Z_k$.
    On $T^*V$, the tautological one-form $\omega_0$ is represented by $\sum_{j=1}^N \zeta_jdg_j(z)$. 
    A closed subvariety of $T_k\times T^*Z_k$ defined by
    \begin{equation}\label{eq:family of likelihood equations}
        W:=\{ (y,x,z,\zeta)\in T_k\times T^*Z_k\mid -\omega_0(z,\zeta)+\nu_0\frac{dy}{y}+\sum_{i=1}^n\nu_i\frac{dx_i}{x_i}=0,\ \ \ 1=yf(x;z)\}
    \end{equation}
    is the support of  $\frac{\Omega^{n}_{(X_k/Z_k)}(X_k)}{\omega\wedge\Omega^{n-1}_{(X_k/Z_k)}(X_k)}$ as an $\mathcal{O}_{T_k\times T^*Z_k}$-module.
    Let $Y$ be a toric compactification of $T_k$ associated to a smooth fan $\Sigma$.
    If the closure $\overline{W}$ of $W$ in $Y\times T^*Z_k$ is $V$ itself, it follows from the coherence of proper direct image sheaves (\cite[Th\'eor\`eme  3.2.1]{EGA}) that $\frac{\Omega^{n}_{(X_k/Z_k)}(X_k)}{\omega\wedge\Omega^{n-1}_{(X_k/Z_k)}(X_k)}$ is a coherent $\mathcal{O}_{T^*Z_k}$-module.
    Let us pick a maximal cone $C$ of $\Sigma$ and assume that $\phi_0,\phi_1,\dots,\phi_n$ be the primitive generators of its rays.
    The affine space associated to $C$ is denoted by $U_C={\rm Spec}\, k[\xi_0,\dots,\xi_n]$.
    Let $\Delta$ (resp. $\Delta^\circ$) denotes a formal disc (resp. formal punctured disc) ${\rm Spec}\, k[\![t]\!]$ (resp. ${\rm Spec}\, k[\![t]\!][t^{-1}]$). 
    The punctured formal disc $\Delta^\circ$ is the generic point of $\Delta$.
    By a well-known argument (see, e.g., \cite[Lemma 2.5]{sattelberger2023maximum}), a closed point $p\in Y\times T^*Z_k$ is in the closure $\overline{W}$ if there exists a morphism $\gamma:\Delta\to Y\times T^*Z_k$ such that $\gamma(\Delta^\circ)\in W$ and $\gamma(0)=p$.
    We pick such a curve $\gamma$ and assume $\gamma(0)\in\cap_{i=0}^a\{ \xi_i=0\}\setminus\cup_{i=a+1}^n\{\xi_i=0\}$ for some $0\leq a\leq n$.
    Then, it follows that the residue of a one-form $\gamma^*(-\omega_0+\nu_0\frac{dy}{y}+\sum_{i=1}^n\nu_i\frac{dx_i}{x_i})$ at $t=0$ must be zero.
    The residue is given by $\langle\sum_{i=0}^a\rho_i\phi_i,\nu\rangle$, where $\rho_i>0$ is the order of $\xi_i\circ\gamma$ at $t=0$ for $i=0,\dots,a$. This is a non-zero element of $k$.
    This is a contradiction and it follows that $\overline{W}=W$.
\end{proof}

\begin{remark}
    The general theory of $D_{[\![\hbar]\!]}$-modules is not parallel to that of usual $D$-modules.
    A difference is that pushforward functor does not preserve coherency.
    For example, the sheaf of meromorphic functions $\mathcal{O}_{\mathbb{C}}(*0)[\![\hbar]\!]$ is not coherent as a $D_{[\![\hbar]\!]}$-module.
\end{remark}

\begin{example}
    We illustrate the role of $\omega_0$ by a simple example.
    Let us consider the case of $Z=\C$.
    We trivialize the cotangent bundle as $T^*Z\simeq\C\times\C={\rm Spec}\,\C[z,\zeta]$.
    Let $f(x;z)$ be $x+z$ and let $J_0$ be the ideal of $K[y,x,z,\zeta]$ generated by 
    $1-yf(x;z),\frac{\nu}{s}-yx,\zeta+sy$.
    It follows from Theorem \ref{thm:3.7} that the quotient $K[y,x,z,\zeta]/J_0$ is a finitely generated $K[z,\zeta]$-module.
    In fact, an identity $y\equiv -\frac{\zeta}{s}$ modulo $J_0$ is easily seen.
    Moreover, 
    $0\equiv f(x;z)(\frac{\nu}{s}-yx)\equiv \frac{\nu}{s}(x+z)-x $, hence $x\equiv\frac{\nu z}{s-\nu}$ modulo $J_0$.
    Thus, $K[y,x,z,\zeta]/J_0\simeq K[z,\zeta]/(z\zeta+s-\nu)$ as $K[z,\zeta]$-modules.
    On the other hand, $K[y,x,z,\zeta]/J_0$ is not a coherent $K[z]$-module.
\end{example}

Since $M_{[\![\hbar]\!]}^{\rm hyp}$ is $\hbar$-torsion free, the canonical morphism $M^{\rm hyp}_{[\![\hbar]\!]}\to M^{\rm hyp}_{((\hbar))}$ is injective.
We equip $M^{\rm hyp}_{[\![\hbar]\!]}$ and $M^{\rm hyp}_{((\hbar))}$ with good filtrations $F_kM^{\rm hyp}_{[\![\hbar]\!]}$ and $F_kM^{\rm hyp}_{((\hbar))}$ so that an identity 
\begin{equation}\label{eq:intersection}
    F_kM^{\rm hyp}_{[\![\hbar]\!]}=F_kM^{\rm hyp}_{((\hbar))}\cap M^{\rm hyp}_{[\![\hbar]\!]}
\end{equation}
holds true.
In fact, let us take a finite presentation $M^{\rm hyp}_{[\![\hbar]\!]}\simeq D_{[\![\hbar]\!]}^\ell/N$ where $N$ is a $D_{[\![\hbar]\!]}$-submodule of a free module $D_{[\![\hbar]\!]}^\ell$.
Since $M^{\rm hyp}_{[\![\hbar]\!]}$ is torsion free by Corollary \ref{cor:no torsion}, it follows that $N=N_{((\hbar))}\cap D_{[\![\hbar]\!]}^\ell$ where $N_{((\hbar))}:=K((\hbar))\otimes_KN$.
Then, we may set $F_kM^{\rm hyp}_{[\![\hbar]\!]}:=(F_kD_{[\![\hbar]\!]})^\ell/N\cap(F_kD_{[\![\hbar]\!]})^\ell$ and $F_kM^{\rm hyp}_{((\hbar))}:=(F_kD_{((\hbar))})^\ell/N\cap(F_kD_{((\hbar))})^\ell$ to obtain the equality \eqref{eq:intersection}.
Let us define a filtration $F_kM_0^{\rm hyp}$ by $
    F_kM_0^{\rm hyp}:=\left(K[\![\hbar]\!]/(\hbar)\right)\otimes_{K[\![\hbar]\!]}F_kM_{[\![\hbar]\!]}^{\rm hyp}\hookrightarrow M_0^{\rm hyp}$.
The last morphism is an injection by \eqref{eq:intersection}.
In fact, if $\overline{m}\in M_{[\![\hbar]\!]}^{\rm hyp}/F_kM_{[\![\hbar]\!]}^{\rm hyp}$ is an $\hbar$-torsion, it follows that $m\in F_kM_{((\hbar))}^{\rm hyp}\cap M_{[\![\hbar]\!]}^{\rm hyp}=F_kM_{[\![\hbar]\!]}^{\rm hyp}$, which proves that $M_{[\![\hbar]\!]}^{\rm hyp}/F_kM_{[\![\hbar]\!]}^{\rm hyp}$ has no $\hbar$-torsion.
By the right-exactness of tensor product, we see that ${\rm gr}_k(M_0^{\rm hyp})=\left(K[\![\hbar]\!]/(\hbar)\right)\otimes_{K[\![\hbar]\!]} {\rm gr}_k(M_{[\![\hbar]\!]}^{\rm hyp})$
and ${\rm gr}(M_0^{\rm hyp})=\left(K[\![\hbar]\!]/(\hbar)\right)\otimes_{K[\![\hbar]\!]} {\rm gr}(M_{[\![\hbar]\!]}^{\rm hyp})$.
In particular, ${\rm gr}(M_0^{\rm hyp})$ is finitely generated.
Therefore, $\{ F_kM_0^{\rm hyp}\}_k$ is a good filtration.
We note that ${\rm gr}(M_{[\![\hbar]\!]}^{\rm hyp})$ has no $\hbar$-torsion.
In fact, if an element $\overline{m}\in {\rm gr}_k(M_{[\![\hbar]\!]}^{\rm hyp})$ satisfies $\hbar^\ell\overline{m}=0$ for some $\ell\geq 1$, it follows that $m\in M_{[\![\hbar]\!]}^{\rm hyp}\cap F_{k-1}M_{((\hbar))}^{\rm hyp}=F_kM_{[\![\hbar]\!]}^{\rm hyp}$ in view of \eqref{eq:intersection}.
This implies that $\overline{m}=0$ in ${\rm gr}_k(M_{[\![\hbar]\!]}^{\rm hyp})$.
Since ${\rm gr}(M_{[\![\hbar]\!]}^{\rm hyp})$ has no $\hbar$-torsion and any element of ${\rm Supp}_0({\rm gr}(M_{[\![\hbar]\!]}^{\rm hyp}))$ is an associated prime, we obtain
\begin{equation}\label{eq:supp0}
{\rm Supp}_0({\rm gr}(M_{((\hbar))}^{\rm hyp}))=\{ \mathfrak{p}\in {\rm Supp}_0({\rm gr}(M_{[\![\hbar]\!]}^{\rm hyp}))\mid\hbar\notin\mathfrak{p}\}={\rm Supp}_0({\rm gr}(M_{[\![\hbar]\!]}^{\rm hyp})).
\end{equation}
Now let us consider an isomorphism $K((\hbar))\otimes_KM^{\rm hyp}\simeq M^{\rm hyp}_{((\hbar))}$ where the field extension $K\hookrightarrow K((\hbar))$ is given by the scaling \eqref{eq:scaling}.
This isomorphism follows immediately by comparing the generators \eqref{eq:generators of J} of $J$ and those of $J_{[\![\hbar]\!]}$ given by \eqref{eq:generators of J_h} through the isomorphism \eqref{eq:DD scaling}.
Thus, by Proposition \ref{prop:defined over C}, it follows that any element $\mathfrak{p}$ in the middle term of \eqref{eq:supp0} is defined over $\C$, hence so is $CC(M_{[\![\hbar]\!]}^{\rm hyp})$.
We need a lemma on the support of a module.

\begin{lemma}\label{lem:identification}
    Let $A$ be a Noetherian commutative ring and let $I\subset A$ be an ideal contained in the Jacobson radical of $A$.
    Let $L$ be a finitely generated $A$-module which is Hausdorff with respect to the $I$-adic topology.
    Then, the following identity holds true:
    \begin{equation}
        {\rm Supp}_{A/I}(L/I\cdot L)\simeq{\rm Supp}_A(L)\cap V_{{\rm Spec}(A)}(I).
    \end{equation}
    Here, the identification is given by the correspondence $\mathfrak{p}_0\mapsto\mathfrak{p}:=\pi^{-1}(\mathfrak{p}_0)$ for any $\mathfrak{p}_0\in {\rm Supp}_{A/I}(L/I\cdot L)$ where $\pi:A\to A/I$ is the quotient map.
\end{lemma}

\begin{proof}
    We set $L_0:=L/I\cdot L$.
    Since $(A/I)_{\mathfrak{p}_0}=A_\mathfrak{p}/I_\mathfrak{p}$ by \cite[Theorem 4.2]{Matsumura}, it follows that the correspondence is well-defined.
    It is obviously injective.
    We prove that it is surjective.
    For any element $\mathfrak{p}\in {\rm Supp}_A(L)\cap V_{{\rm Spec}(A)}(I)$, we set $\mathfrak{p}_0:=\pi(\mathfrak{p})$.
    Suppose that $(L_0)_{\mathfrak{p}_0}=0$.
    Then, it follows that $L_\mathfrak{p}=I_\mathfrak{p}\cdot L_\mathfrak{p}$, hence $L_\mathfrak{p}=\cap_{n>0}I_\mathfrak{p}^n\cdot L_\mathfrak{p}=0$.
    The last equality follows from the definition of the Jacobson radical and \cite[Theorem 8.9]{Matsumura}.
\end{proof}

By Lemma \ref{lem:identification} and \eqref{eq:supp0}, there is a bijection ${\rm Supp}_0({\rm gr}(M^{\rm hyp}_0))\ni\mathfrak{p}_0\mapsto \pi^{-1}(\mathfrak{p}_0)=\mathfrak{p}\in{\rm Supp}_0({\rm gr}(M^{\rm hyp}_{[\![\hbar]\!]}))$.
Since any $\mathfrak{p}\in{\rm Supp}_0({\rm gr}(M^{\rm hyp}_{[\![\hbar]\!]}))$ is defined over $\C$, it follows that we can identify ${\rm Supp}_0({\rm gr}(M^{\rm hyp}_0))$ with ${\rm Supp}_0({\rm gr}(M^{\rm hyp}_{[\![\hbar]\!]}))$.
Thus, we can identify three objects ${\rm Supp}_0({\rm gr}(M^{\rm hyp}_0))$, ${\rm Supp}_0({\rm gr}(M^{\rm hyp}_{((\hbar))}))$ and ${\rm Supp}_0({\rm gr}(M^{\rm hyp}_{[\![\hbar]\!]}))$.
To prove Theorem \ref{thm:main result}, it remains to study multiplicities.

\begin{proposition}
    For any $\mathfrak{p}\in {\rm Supp}_0({\rm gr}(M^{\rm hyp}_{[\![\hbar]\!]}))$, one has the following equalities:
    $$
    {\rm length}_{\mathcal{O}_{T^*Z,\mathfrak{p}_0}}({\rm gr}(M^{\rm hyp}_{0})_{\mathfrak{p}_0})
    =
    {\rm length}_{\mathcal{O}_{T^*Z_{[\![\hbar]\!]},\mathfrak{p}}}({\rm gr}(M^{\rm hyp}_{[\![\hbar]\!]})_\mathfrak{p})
    =
    {\rm length}_{\mathcal{O}_{T^*Z_{((\hbar))},\mathfrak{p}}}({\rm gr}(M^{\rm hyp}_{((\hbar))})_\mathfrak{p}).
    $$
\end{proposition}

\begin{proof}
    Let $0\subset N_1\subset \cdots \subset N_\ell={\rm gr}(M^{\rm hyp}_{[\![\hbar]\!]})_\mathfrak{p}$ be a composition series of ${\rm gr}(M^{\rm hyp}_{[\![\hbar]\!]})_\mathfrak{p}$.
    Since $\mathcal{O}_{T^*Z,\mathfrak{p}}$ is a Noetherian local ring, it follows that $N_{j}/N_{j-1}\simeq \mathcal{O}_{T^*Z_{[\![\hbar]\!]},\mathfrak{p}}/\mathfrak{p}\mathcal{O}_{T^*Z_{[\![\hbar]\!]},\mathfrak{p}}\simeq K[\![\hbar]\!]\otimes_K\mathcal{O}_{T^*Z,\mathfrak{p}_0}/\mathfrak{p}_0\mathcal{O}_{T^*Z,\mathfrak{p}_0}$
    for any $j=1,\dots,\ell$.
    Since ${\rm Tor}^{K[\![\hbar]\!]}_1(N_j/N_{j-1},K[\![\hbar]\!]/(\hbar))=0$, the natural morphism $(K[\![\hbar]\!]/(\hbar))\otimes_{K[\![\hbar]\!]}N_{j-1}\to (K[\![\hbar]\!]/(\hbar))\otimes_{K[\![\hbar]\!]}N_j$ is an injection and that $(K[\![\hbar]\!]/(\hbar))\otimes_{K[\![\hbar]\!]}N_j/N_{j-1}\simeq \mathcal{O}_{T^*Z,\mathfrak{p}_0}/\mathfrak{p}_0\mathcal{O}_{T^*Z,\mathfrak{p}_0}$.
    Hence, $0\subset (K[\![\hbar]\!]/(\hbar))\otimes_{K[\![\hbar]\!]}N_1\subset \cdots \subset (K[\![\hbar]\!]/(\hbar))\otimes_{K[\![\hbar]\!]}N_\ell={\rm gr}(M^{\rm hyp}_{0})_{\mathfrak{p}_0}$ is a composition series of ${\rm gr}(M^{\rm hyp}_{0})_{\mathfrak{p}_0}$.
    Likewise, $0\subset K((\hbar))\otimes_{K[\![\hbar]\!]}N_1\subset \cdots \subset K((\hbar))\otimes_{K[\![\hbar]\!]}N_\ell={\rm gr}(M^{\rm hyp}_{((\hbar))})_\mathfrak{p}$ is a composition series of ${\rm gr}(M^{\rm hyp}_{((\hbar))})_\mathfrak{p}$.
\end{proof}

\subsection{A geometric description of Euler discriminant}\label{subsec:geometric description}

Let us first consider the homogenized cotangent bundle of $Z$.
This is simply defined as a direct sum $E:=\underline{\C}_Z\oplus T^*Z$ where $\underline{\C}_Z$ is the trivial bundle of rank one over $Z$.
Let $0_E$ be the global zero section of $E$.
The homogenized cotangent bundle is defined as the quotient bundle $\mathbb{P}^*Z:=(E\setminus 0_E)//\C^*$ where $\C^*$ acts by fiberwise scaling.
The symbol $\pi_{\mathbb{P}^*Z}:\mathbb{P}^*Z{\to} Z$ denotes the associated projection.
The morphism $T^*Z\ni(z;\zeta)\mapsto (z;1,\zeta)\in E$ naturally induces a morphism $\iota_{T^*Z}:T^*Z\to \mathbb{P}^*Z$.
A point in $\mathbb{P}^*Z$ is symbolically denoted by $(z;[\tau:\zeta])$.
We write $H_\infty\subset \mathbb{P}^*Z$ for the "hyperplane at infinity" in the fiber direction defined by $H_\infty:=\{\tau=0\}$.
$H_\infty$ is also identified with a quotient $(T^*Z\setminus 0_{T^*Z})//\C^*$ where $0_{T^*Z}$ is the global zero section of $T^*Z$ and $\C^*$ acts by fiberwise scaling.
Let $\pi_{T^*Z}:T^{n+1}\times T^*Z{\to} T^*Z$ be the natural projection.
To simplify the notation, we write $V_0$ for the closure of ${\pi_{T^*Z}(V(J_0))}$ in $T^*Z$.
We write $\mathbb{P}{\rm Char}(M^{\rm hyp})$ for the image of ${\rm Char}(M^{\rm hyp})\setminus 0_{T^*Z}$ in $H_\infty\simeq (T^*Z\setminus 0_{T^*Z})//\C^*$.
In view of a standard argument on initial ideals (\cite[Theorem 15.17]{eisenbud2013commutative}) and Theorem \ref{thm:main result}, the variety $\mathbb{P}{\rm Char}(M^{\rm hyp})$ is identified with $\overline{\iota_{T^*Z}(V_0)}\cap H_\infty$ where the closure $\overline{\bullet}$ is taken in $\mathbb{P}^*Z$.
By the definition of the singular locus and Corollary \ref{cor:ED=Sing}, we obtain the following description.

\begin{proposition}\label{prop:ED formula}
$$\nabla^\pi_\chi(Z)=\pi_{\mathbb{P}^*Z}\left(\overline{\iota_{T^*Z}(V_0)}\cap H_\infty\right).$$ 
\end{proposition}

Let $\mathcal{F}$ be the vanishing locus of a one-form
\begin{equation}\label{eq:likelihood one-form}
    -\omega_0-\nu_0\tau\frac{df}{f}+\sum_{i=1}^n\tau\nu_i\frac{dx_i}{x_i}
    \end{equation}
on $T^n\times\mathbb{P}^*Z$ and let $Y$ be any smooth toric compactification of $T^n$.
Let $p_1={\rm id}_Y\times\pi_{\mathbb{P}^*Z}:Y\times\mathbb{P}^*Z\to Y\times Z$ and $p_Z:Y\times Z\to Z$ be projections.
We set $V_f:=\overline{V_{T^n\times Z}(f)}$, where the closure is taken in $Y\times Z$.
In \cite[\S 7 item 4]{telen2024euler}, the authors conjectured an identity
\begin{equation}\label{eq:TW identity}
 \nabla_\chi^\pi(Z)=p_Z(\overline{p_1(\mathcal{F})}\cap V_f).
\end{equation}
We deduce \eqref{eq:TW identity} from a stronger result \eqref{eq:7.2}.
In particular, Corollary \ref{cor:ED=Sing} and \eqref{eq:TW identity} resolve two questions raised in \cite[\S7, item 4 and item 7]{telen2024euler}: $p_Z\left(\overline{p_1(\mathcal{F})}\cap V_f\right)$ does not depend on the choice of a toric compactification $Y$ and it is the singular locus of a $D$-module corresponding to an Euler integral.
Let $p_2:Y\times\mathbb{P}^*Z\to\mathbb{P}^*Z$ be the natural projection.

\begin{theorem}
The following identity holds true:
    \begin{equation}\label{eq:7.2}
    \mathbb{P}{\rm Char}(M^{\rm hyp})=
    p_2(\overline{\mathcal{F}}\cap p_1^{-1}(V_f)).
\end{equation}
Here, the closure $\overline{\mathcal{F}}$ is taken in the product variety $Y\times\mathbb{P}^*Z$.
\end{theorem}

\begin{proof}
    
Let us confirm that \eqref{eq:7.2} implies \eqref{eq:TW identity}.
We first note that $\overline{p_1(\mathcal{F})}=p_1(\overline{\mathcal{F}})$ since $p_1$ is a closed map.
This also implies that
$\overline{p_1(\mathcal{F})}\cap V_f=p_1(\overline{\mathcal{F}}\cap p_1^{-1}(V_f))$.
Thus, it follows that $\nabla_\chi^\pi(Z)=\pi_{\mathbb{P}^*Z}(p_2(\overline{\mathcal{F}}\cap p_1^{-1}(V_f)))=p_Z(\overline{p_1(\mathcal{F})}\cap V_f)$, which proves \eqref{eq:TW identity}.

On the other hand, \eqref{eq:7.2} follows from the following identity.
\begin{equation}\label{eq:old lemma}
\overline{\mathcal{F}}\cap p_1^{-1}(V_f)=\overline{\mathcal{F}}\cap p_2^{-1}(H_\infty). 
\end{equation}
In fact, \eqref{eq:old lemma} and the closedness of the morphism $p_2:Y\times\mathbb{P}^*Z\to\mathbb{P}^*Z$ imply a sequence of identities $\mathbb{P}{\rm Char}(M^{\rm hyp})=\overline{p_2(\mathcal{F})}\cap H_\infty =p_2(\overline{\mathcal{F}}\cap p_2^{-1}(H_\infty))=
    p_2(\overline{\mathcal{F}}\cap p_1^{-1}(V_f))$.

Let us prove \eqref{eq:old lemma}.
    The inclusion $\overline{\mathcal{F}}\cap p_1^{-1}(V_f)\supset\overline{\mathcal{F}}\cap H_\infty$ is because $\tau=f(x;z)$ on $\overline{\mathcal{F}}$ and $H_\infty=\{\tau=0\}$.
    We show the other inclusion.
    We pick a toric coordinate $\xi=(\xi_1,\dots,\xi_n)$ associated to a cone spanned by $\phi_1,\dots,\phi_n$.
    We set 
    $\rho_i=\underset{v\in {\rm New}(f)}{\rm min}\langle\phi_i,v\rangle$ for $i=1,\dots,n$.
    Let $\gamma:\Delta\to \overline{\mathcal{F}}$ be a morphism such that $\gamma(0)\in p_1^{-1}(V_f)\cap\cap_{i=1}^k\{\xi_i=0\}\setminus\cup_{j=k+1}^n\{\xi_j=0\}$ for some $k$.
    In this coordinate, $f$ is written as 
    $
    f=\prod_{i=1}^k\xi_i^{\rho_i}\tilde{f}(\xi)=\prod_{i=1}^k\xi_i^{\rho_i}\left(g_0(\xi)+\sum_{i=1}^k\xi_ig_i(\xi)\right),
    $
    where $\tilde{f},g_1,\dots,g_k$ are polynomials in $\xi$ with $g_0|_{\xi_1=\cdots=\xi_k=0}\not\equiv 0$ and $\rho_i$ are integers.
    The one-form \eqref{eq:likelihood one-form} takes the following form:
    $$
    -\omega_0-\nu_0\tau\frac{d\tilde{f}}{\tilde{f}}+\sum_{i=1}^k\tau(\langle\phi_i,\nu\rangle-\nu_0\rho_i)\frac{d\xi_i}{\xi_i}
    +\sum_{i=k+1}^n\tau\langle\phi_i,\nu\rangle\frac{d\xi_i}{\xi_i}.
    $$
    As in the last part of Theorem \ref{thm:3.7}, it follows that ${\rm ord}(\tau\circ\gamma)>0$, which implies $\gamma(0)\in H_\infty$.

\end{proof}

\subsection{Gelfand-Kapranov-Zelevinsky case}\label{subsec:GKZ case}

Let us first recall the definition of GKZ system.
In this section, we consider a special family of polynomials $f(x;z)=\sum_{a\in A}z_ax^a$ where $A\subset \Z^n$ is a finite subset and $z=(z_a)_{a\in A}\in\C^A=Z$.
We recycle the notation used in Example \ref{ex:GKZ}.
Let $A^h$ be the homogenization of $A$, i.e., $A^h$ is a collection of lattice vectors $(1,a)\in\Z^{n+1}$ with $a\in A$.
Throughout this section, we assume that $A^h$ spans $\mathbb{Q}^{n+1}$ over $\mathbb{Q}$.
Let $R_A\subset D_{Z_K}$ be the subring of $D_{Z_K}$ generated by $\partial_a$ for $a\in A$ over $K$.
Let $I_A$ be the toric ideal of the configuration $A^h$.
Any element of $I_A$ is a binomial $\prod_{a\in A}\partial_a^{u_a}-\prod_{a\in A}\partial_a^{v_a}$ of $\partial_a$.
Here, $u_a,v_a\in\Z_{\geq 0}$ satisfy a relation $\sum_{a\in A}u_a=\sum_{a\in A}v_a$ and $\sum_{a\in A}u_aa=\sum_{a\in A}v_aa$.
We set
$E_0:=\sum_{a\in A}z_a\partial_a$ and $E_i:=\sum_{a\in A}a_iz_a\partial_a$ for $i=1,\dots,n$, where $a_i$ is the $i$-th coordinate of $a\in A$.
The GKZ ideal $H_A(\nu)\subset D_{Z_K}$ is a left ideal generated by $E_i+\nu_i$ for $i=0,\dots,n$ and elements of $I_A$.
The GKZ system $M_A(\nu)$ is the left $D_{Z_K}$-module $D_{Z_K}/H_A(\nu)$.
In the same manner, we define an $\hbar$-version of GKZ system.
Let $I_A^{\,\hbar}$ be the left ideal generated by binomials of the form $\prod_{a\in A}(\partial_a^\hbar)^{u_a}-\prod_{a\in A}(\partial_a^\hbar)^{v_a}$, where $u_a,v_a\in\Z_{\geq 0}$ satisfy a relation $\sum_{a\in A}u_a=\sum_{a\in A}v_a$ and $\sum_{a\in A}u_aa=\sum_{a\in A}v_aa$.
We define $E_i^\hbar\in D_{[\![\hbar]\!]}$ for $i=0,\dots,n$ by
$E_0^\hbar:=\sum_{a\in A}z_a\partial_a^\hbar$ and  $E_i^\hbar:=\sum_{a\in A}a_iz_a\partial_a^\hbar$ for $i=1,\dots,n$.
An $\hbar$-version of GKZ ideal $H_A^\hbar(\nu)\subset D_{[\![\hbar]\!]}$ is a left ideal generated by $E_i^\hbar+\nu_i$ for $i=0,\dots,n$ and elements of $I_A^{\,\hbar}$.
An $\hbar$-version of GKZ system $M_A^\hbar(\nu)$ is the left $D_{[\![\hbar]\!]}$-module $D_{[\![\hbar]\!]}/H_A^\hbar(\nu)$.
The purpose of this subsection is to prove the following equality.

\begin{theorem}\label{thm:hbar GKZ}
    There is an isomorphism of $D_{[\![\hbar]\!]}$-modules $M_A^\hbar(\nu)\to M^{\rm hyp}_{[\![\hbar]\!]}$ which sends the class $[1]$ in $M_A^\hbar(\nu)$ to that in $M^{\rm hyp}_{[\![\hbar]\!]}$.
    In particular, the class $[1]$ in $M^{\rm hyp}_{[\![\hbar]\!]}$ is a cyclic generator of $M^{\rm hyp}_{[\![\hbar]\!]}$ as a $D_{[\![\hbar]\!]}$-module.
\end{theorem}

In the following, we state two lemmas necessary to prove Theorem \ref{thm:hbar GKZ} from \cite{adolphson1994hypergeometric} and \cite{schulze2009hypergeometric}.
We do not give proofs to the lemmas below as they are $\hbar$-analogues of existing propositions.
We specify the corresponding propositions below.

\begin{remark}
    It is also possible to introduce and study the usual static version of GKZ system.
    Namely, one can take $\nu_0,\dots,\nu_n$ to be any complex numbers.
    For example, a parameter $\nu$ is said to be s-resonant if $\nu\in (1,a)+{\rm qdeg}(S_A/\langle\zeta_a\rangle)$ for some $a\in A$, where ${\rm qdeg}(S_A/\langle\zeta_a\rangle)$ denotes the set of quasi-degrees \cite[Definition 5.2]{Matusevich-Miller-Walther}.
    Then, the static version of Lemma \ref{lem:non-resonance} is the following: $\nu$ is s-resonant if and only if $\partial_a^\hbar:\mathcal{K}_\bullet(E^\hbar+\nu,S_A)\to \mathcal{K}_\bullet(E^\hbar+\nu,S_A)$ is not a quasi-isomorphism for some $a\in A$.
    This is not pursued to keep the exposition consistent with other sections.
\end{remark}

The ring $R_A$ is naturally $\Z^{n+1}$-graded by setting ${\rm deg} (\partial_a)=-(1,a)$, which induces a $\Z^{n+1}$-grading of $S_A:=R_A/I_A$.
We also introduce a $\Z^{n+1}$-grading to the Weyl algebra $D_{[\![\hbar]\!]}$ by setting ${\rm deg} (\partial_a^\hbar):=-(1,a)$, ${\rm deg}(z_a):=(1,a)$ and ${\rm deg}(\hbar):=(0,0)$.
For a homogeneous element $m$ of a $\Z^{n+1}$-graded $D_{[\![\hbar]\!]}$-module $M$, we set
\begin{equation}\label{eq:EK action}
    (E_i^\hbar+\nu_i)\circ m:=(E_i^\hbar+\nu_i+\deg(m)\hbar)m.
\end{equation}
We identify $R_A$ as a subring of $D_{[\![\hbar]\!]}$ generated by $\partial^\hbar_a$ by the correspondence $\partial_a\mapsto \partial^\hbar_a$.
For a $\Z^{n+1}$-graded $R_A$-module $N$, the tensor product $D_{[\![\hbar]\!]}\otimes_{R_A}N$ is naturally equipped with a structure of a $\Z^{n+1}$-graded $D_{[\![\hbar]\!]}$-module.
The {\it Euler-Koszul complex} $\mathcal{K}_\bullet(E^\hbar+\nu,N)$ 
is the Koszul complex of left $D_{[\![\hbar]\!]}$-modules defined by the sequence $E^\hbar +\nu$ of commuting endomorphisms on the left $D_{[\![\hbar]\!]}$-module $D_{[\![\hbar]\!]}\otimes_{R_A}N$ concentrated in homological degrees $n+1$ to $0$.
The $i$-th Euler–Koszul homology of $N$ is $
\mathcal{H}_i(E^\hbar +\nu;N) := H_i(\mathcal{K}_\bullet(E^\hbar +\nu;N))$.
The following lemmas are $\hbar$-versions of \cite[Theorem 4.4]{adolphson1994hypergeometric} and \cite[Theorem 3.6]{schulze2009hypergeometric}.

\begin{lemma}\label{lem:hbar GKZ isom}
    There is an isomorphism $M_A^\hbar(\nu)\simeq \mathcal{H}_0^\hbar(E^\hbar+\nu,S_A)$ of left $D_{[\![\hbar]\!]}$-modules.
    It sends the equivalence class of unity of $M_A^\hbar(\nu)$ to that of $\mathcal{H}_0^\hbar(E^\hbar+\nu,S_A)$.
\end{lemma}

\begin{lemma}\label{lem:non-resonance}
    The morphism $\partial_a^\hbar:\mathcal{K}_\bullet(E^\hbar+\nu,S_A)\to \mathcal{K}_\bullet(E^\hbar+\nu,S_A)$ is a quasi-isomorphism for any $a\in A$.
\end{lemma}
Since the cotangent bundle $T^*Z_K$ is trivialized as $Z_K\times (Z_K)^*$ where $*$ denotes the dual vector space, the coordinate ring $\mathcal{O}_{T^*Z_K}$ is also trivialized as $\mathcal{O}_{T^*Z_K}=K[z_a,\zeta_a;a\in A]$.
We set ${\rm deg}(\zeta_a):=-(1,a)$ and ${\rm deg}(z_a):=(1,a)$ so that $\mathcal{O}_{T^*Z_K}$ is a $\Z^{n+1}$-graded module.
We identify $R_A$ as a subring of $\mathcal{O}_{T^*Z_K}$ generated by $\zeta_a$ by the correspondence $\partial_a\mapsto \zeta_a$.
We set
$E_0^0:=\sum_{a\in A}z_a\zeta_a$ and $E_i^0:=\sum_{a\in A}a_iz_a\zeta_a$ for $i=1,\dots,n$, where $a_i$ is the $i$-th coordinate of $a\in A$.
Given a $\Z^{n+1}$-graded $S_A$ module $M$, we write $\mathcal{K}_\bullet (E^0+\nu,M)$ for the Koszul complex of the operators $(E_i^0+\nu_i):\mathcal{O}_{T^*Z_K}\otimes_{R_A}M\to \mathcal{O}_{T^*Z_K}\otimes_{R_A}M$ for $i=0,\dots,n$.
The $i$-th Koszul homology of $N$ is $
\mathcal{H}_i(E^0 +\nu;N) := H_i(\mathcal{K}_\bullet(E^0 +\nu;N))$.

Let us now prove Theorem \ref{thm:hbar GKZ}.
By Lemma \ref{lem:hbar GKZ isom} and Lemma \ref{lem:non-resonance}, there is an isomorphism of $D_{[\![\hbar]\!]}$-modules $M_A^\hbar(\nu)\simeq \mathcal{H}_0(E^\hbar+\nu,S_A[\partial_a^{-1};a\in A])$.
We consider a short exact sequence 
$0\to \mathcal{K}_\bullet(E^\hbar+\nu,S_A[\partial_a^{-1};a\in A])\overset{\hbar}{\to} \mathcal{K}_\bullet(E^\hbar+\nu,S_A[\partial_a^{-1};a\in A])\to \mathcal{K}_\bullet(E^0+\nu,S_A[\partial_a^{-1};a\in A])\to 0$.
By taking its homology groups, we obtain the following exact sequence:
\begin{equation}\label{eq:no torsion for GKZ}
    \mathcal{H}_1(E^0+\nu,S_A[\partial_a^{-1};a\in A])\to M_A^\hbar(\nu)\overset{\hbar}{\to}M_A^\hbar(\nu).
\end{equation}
We claim that $\{ E^0_i+\nu\}_{i=0}^n$ is a regular sequence of $\mathcal{O}_{T^*Z_K}\otimes_RS_A[\partial_a^{-1};a\in A]$.
To see this, it is enough to show that $\mathcal{O}:=\mathcal{O}_{T^*Z_K}\otimes_RS_A[\partial_a^{-1};a\in A]/\sum_{i=0}^n\mathcal{O}_{T^*Z_K}\otimes_RS_A[\partial_a^{-1};a\in A](E_i+\nu_i)$ is a smooth $\C$-algebra of dimension $|A|-n-1$.
Let us regard $A^h$ as a $(n+1)\times |A|$ matrix, $z$ as a $|A|$-dimensional column vector, and regard $\nu$ as an $(n+1)$-dimensional column vector.
${\rm Spec}\,\mathcal{O}\subset {\rm Spec}\,\mathcal{O}_{T^*Z_K}\otimes_RS_A[\partial_a^{-1};a\in A]$ is defined by a constraint $A^h{\rm diag}(\zeta_a)_{a\in A}\cdot z=-\nu$, which is always solvable as an equation for $z$ since the matrix $A^h{\rm diag}(\zeta_a)_{a\in A}$ has rank $n+1$.
Thus, $\mathcal{O}$ is the coordinate ring of a rank $(|A|-n-1)$-vector bundle on ${\rm Spec}\,S_A[\zeta^{-1}]\simeq T_K$.
It follows that $\mathcal{H}_1(E^0+\nu,S_A[\partial_a^{-1};a\in A])=0$, hence $M_A^\hbar(\nu)$ has no $\hbar$-torsion by \eqref{eq:no torsion for GKZ}.
Now the same proof as that of Theorem \ref{thm:main result} shows that $CC(M_A^\hbar(\nu))=CC(M_A(\nu))$.
Thus, we obtain $CC(M_A^\hbar(\nu))=CC(M_A(\nu))=CC(M^{\rm hyp}_{[\![\hbar]\!]})=E^{\rm hyp}$.
Let $N$ be the kernel of the morphism $M_A^\hbar(\nu)\to M^{\rm hyp}_{[\![\hbar]\!]}$.
Since the characteristic cycle is additive, we obtain that $CC(N)=0$.
It follows that $N=0$.
\qed

A version of Theorem \ref{thm:main result} in the context of this subsection is the following corollary.
Note that the characteristic cycle of GKZ system is already studied in detail in \cite{gel1989hypergeometric} and \cite{Schulze_Walther}.

\begin{corollary}
    Let $I_0$ be the ideal of $\mathcal{O}_{T^*\C^A}$ generated by $I_A$ and $E^0_i$ for $i=0,\dots,n$.
    Then, $CC(M_A(\nu))$ is the algebraic cycle associated to ${\rm gr}(\mathcal{O}_{T^*\C^A}/I_0)$.
\end{corollary}

\section{A hypergeometric view on Landau analysis}\label{sec:Landau analysis}

We assume $Z$ is an affine space in this section and $\mathcal{O}_{T^*Z}$ denotes the coordinate ring of $T^*Z$, which is a polynomial ring.
Let us set $I_0:=\mathcal{O}_{T^*Z}\cap J_0$.
The quotient $\mathcal{O}_{T^*Z}/I_0$ is an $\mathcal{O}_{T^*Z}$-submodule of $M^{\rm hyp}_0$.
Let $E^{\rm sub}$ denote the characteristic cycle of $\mathcal{O}_{T^*Z}/I_0$.
It is a sub-cycle of $E^{\rm hyp}$ whose support is same as the characteristic variety of $M^{\rm hyp}$.
We note that any irreducible component of the characteristic variety of $M^{\rm hyp}$ is a conic Lagrangian subvariety of $T^*Z$, hence it is a conormal bundle $T^*_Y Z$ of an irreducible closed subvariety $Y\subset Z$.
Let us write $E^{\rm hyp}=\sum_Y m_Y T^*_Y Z\ \ \ \text{and}\ \ \ E^{\rm sub}=\sum_Y m'_Y T^*_Y Z$.
Then, it follows by construction that $0<m'_Y\leq m_Y$ for any $Y$ occurring in the summation.
There are examples in which a strict inequality $m'_Y<m_Y$ holds true.
An example is $Z=\C={\rm Spec}\, \C[z]$ and $f(x;z)=z f(x)$ where $f(x)$ is a Laurent polynomial such that the absolute value of the Euler characteristic of $V_{T}(f)$ is strictly larger than one.

In many examples, however, it is possible to check whether $m_Y=m'_Y$ or not.
A way to confirm it is through Kashiwara's index formula.
For a closed subvariety $Y\subset Z$, let ${\rm Eu}_Y:Z\to\Z$ denotes the Euler obstruction of $Y$ (\cite[Definition 4.1.36]{dimca2004sheaves}).
Kashiwara's index formula applied to our setup proves an identity 
\begin{equation}\label{eq:Euler index}
\chi_z=\sum_{Y}(-1)^{{\rm codim}Y}m_Y {\rm Eu}_Y(z)    
\end{equation}
for any $z\in Z$.
Computing $\chi_z$ and ${\rm Eu}_Y(z)$ for various $z$, we can recover the multiplicity $m_Y$.
Computing the Euler obstruction is, however, a challenging problem in general.
We point out that it can be computed for an isolated singularity \cite[Example 2]{kashiwara1973index} and for an affine toric variety \cite[Theorem 4.7]{matsui2011geometric}.
An algorithm for computing Euler obstruction in terms of the maximal likelihood degree at a point is also proposed in \cite{rodriguez2020computing}.
For any $\mathfrak{p}\in {\rm Supp}_0({\rm gr}(M))$, one has an identity $V(\mathfrak{p})=T_{\pi_{T^*Z}(V(\mathfrak{p}))}^*Z$.
This is because $V(\mathfrak{p})\subset T^*Z$ is an irreducible conic Lagrangian subvariety.
Therefore, $E^{\rm hyp}$ is determined by a list $\{(\pi_{T^*Z}(V(\mathfrak{p})),m_\mathfrak{p})\}_{\mathfrak{p}\in {\rm Supp}_0({\rm gr}(M))}$.

In quantum field theory,the Euler integral \eqref{eq:Euler Integral intro} appears as the Lee-Pomeransky representation of a Feynman integral.
In this setting, $f(x;z)$ is the graph polynomial $f_G(x;z)$ associated with a Feynman graph $G$. 
The variables $x$ are called {\it Schwinger parameters}, while $z\in Z$ represent the {\it kinematic data}.
These terminologies are standard and not recalled here.
See \cite[Chapter 2]{weinzierl2022feynman} for this account.
In the Feynman diagrams below, edge labels are indicated by putting the $i$-th Schwinger parameter $x_i$ to the $i$-th internal edge.
We write $E^{\rm hyp}_G$ for the hypergeometric discriminant associated to a family of polynomials $f_G(x;z)$.

\begin{minipage}[b]{0.33\textwidth}
\centering
\begin{tikzpicture}
    \begin{feynhand}
    \vertex (a) at (0,0);
    \vertex (b) at (2,0);
    \vertex [particle] (p1) at (-1,0){$p$};
    \vertex [particle] (p2) at (3,0){$-p$};
    \propag (p1) to (a);
    \propag (p2) to (b);
    \propag [plain] (a) to [out=90, in=90, edge label=$x_1$](b);
    \propag [plain] (a) to [out=-90, in=-90](b);
    \node at (1,-0.9){$x_2$};
    \end{feynhand}
\end{tikzpicture}
\captionsetup{font=footnotesize}
\captionof{figure}{Bubble $\texttt{B}_2$}
\end{minipage}
\begin{minipage}[b]{0.33\textwidth}
\centering
\begin{tikzpicture}
    \begin{feynhand}
    \vertex (a) at (0,0);
    \vertex (b) at (2,0);
    \vertex [particle] (p1) at (-1,0){$p$};
    \vertex [particle] (p2) at (3,0){$-p$};
    \propag (p1) to (a);
    \propag (p2) to (b);
    \propag [plain] (a) to [out=90, in=90, edge label=$x_1$](b);
    \propag [plain] (a) to [out=-90, in=-90](b);
    \propag [plain] (a) to (b);
    \end{feynhand}
    \node at (1,-.2){$x_2$};
    \node at (1,-1){$x_3$};
    \end{tikzpicture}
\captionsetup{font=footnotesize}
\captionof{figure}{Sunrise $\texttt{B}_3$}
\end{minipage}
\begin{minipage}[b]{0.33\textwidth}
\centering
\begin{tikzpicture}
    \begin{feynhand}
    \vertex (a) at (0,0);
    \vertex (b) at (2,0);
    \vertex [particle] (p1) at (-1,0){$p$};
    \vertex [particle] (p2) at (3,0){$-p$};
    \propag (p1) to (a);
    \propag (p2) to (b);
    \propag [plain] (a) to [out=90, in=90, edge label=$x_1$](b);
    \propag [plain] (a) to [out=-90, in=-90](b);
    \propag [plain] (a) to [out=-45, in=-135](b);
    \propag [plain] (a) to [out=45, in=135](b);
    \end{feynhand}
    \node at (1,.2){$x_2$};
    \node at (1,-.2){$x_3$};
    \node at (1,-1){$x_4$};
    \end{tikzpicture}
\captionsetup{font=footnotesize}
\captionof{figure}{$\texttt{B}_4$}
\end{minipage}

\begin{minipage}[b]{0.23\textwidth}
\centering
 \begin{tikzpicture}
        \begin{feynhand}
        \vertex (a) at (1,1);
        \vertex (b) at (1,3);
        \vertex (c) at (3,2);
        \vertex [particle] (p1) at (0,0){$p_1$};
        \vertex [particle] (p2) at (0,4){$p_2$};
        \vertex [particle] (p3) at (4,2){$p_3$};
        \propag (p1) to (a);
        \propag (p2) to (b);
        \propag (p3) to (c);
        \propag [plain] (a) to [out=180, in=180, edge label=$x_2$](b);
        \propag [plain] (a) to [edge label'=$x_1$](b);
        \propag [plain] (b) to [edge label=$x_4$](c);
        \propag [plain] (c) to [edge label=$x_3$](a);
        \end{feynhand}
\end{tikzpicture}
\captionsetup{font=footnotesize}
\captionof{figure}{Parachute \texttt{par}}
\end{minipage}
\begin{minipage}[b]{0.23\textwidth}
\centering
\begin{tikzpicture}
\begin{feynhand}
\vertex (a) at (-1,1);
\vertex (b) at (1,1);
\vertex (c) at (1,-1);
\vertex (d) at (-1,-1);
\vertex [particle] (p1) at (-1.7,1.7){$p_1$};
\vertex [particle] (p4) at (1.7,1.7){$p_4$};
\vertex [particle] (p3) at (1.7,-1.7){$p_3$};
\vertex [particle] (p2) at (-1.7,-1.7){$p_2$};
\propag (p1) to (a);
\propag (p4) to (b);
\propag (p3) to (c);
\propag (p2) to (d);
\propag [plain] (a) to (b);
\propag [plain] (b) to (c);
\propag [plain] (c) to (d);
\propag [plain] (d) to (a);
\propag [plain] (b) to (d);
\draw[-] (a) -- (-.1,.1);
\draw[-] (c) -- (.1,-.1);
\node at (-1.3,0){$x_1$};
\node at (0,-1.2){$x_2$};
\node at (1.3,0){$x_3$};
\node at (0,1.3){$x_4$};
\node at (-.3,0.6){$x_5$};
\node at (-.6,-0.3){$x_6$};
\end{feynhand}
\end{tikzpicture}
\captionsetup{font=footnotesize}
\captionof{figure}{Non-planar box \texttt{env}}
\label{fig:env}
\end{minipage}
\begin{minipage}[b]{0.23\textwidth}
\centering
\begin{tikzpicture}
\begin{feynhand}
\vertex (a) at (-1,1);
\vertex (b) at (1,1);
\vertex (c) at (1,-1);
\vertex (d) at (-1,-1);
\vertex [particle] (p1) at (-1.7,1.7){$p_1$};
\vertex [particle] (p4) at (1.7,1.7){$p_4$};
\vertex [particle] (p3) at (1.7,-1.7){$p_3$};
\vertex [particle] (p2) at (-1.7,-1.7){$p_2$};
\propag (p1) to (a);
\propag (p4) to (b);
\propag (p3) to (c);
\propag (p2) to (d);
\propag [plain] (a) to (b);
\propag [plain] (b) to (c);
\propag [plain] (c) to (d);
\propag [plain] (d) to (a);
\propag [plain] (b) to (d);
\node at (-1.3,0){$x_1$};
\node at (0,-1.2){$x_2$};
\node at (1.3,0){$x_3$};
\node at (0,1.3){$x_4$};
\node at (-0.4,0){$x_5$};
\end{feynhand}
\end{tikzpicture}
\captionsetup{font=footnotesize}
\captionof{figure}{Slashed box \texttt{acn}}
\label{fig:acn}
\end{minipage}
\begin{minipage}[b]{0.23\textwidth}
\centering
\begin{tikzpicture}
\begin{feynhand}
\vertex (a) at (-1,1);
\vertex (b) at (1,1);
\vertex (c) at (1,-1);
\vertex (d) at (-1,-1);
\vertex [particle] (p1) at (-1.7,1.7){$p_1$};
\vertex [particle] (p4) at (1.7,1.7){$p_4$};
\vertex [particle] (p3) at (1.7,-1.7){$p_3$};
\vertex [particle] (p2) at (-1.7,-1.7){$p_2$};
\propag (p1) to (a);
\propag (p4) to (b);
\propag (p3) to (c);
\propag (p2) to (d);
\propag [plain] (a) to [out=-135, in=135, edge label'=$x_1$](d);
\propag [plain] (b) to [edge label=$x_4$](c);
\propag [plain] (c) to [edge label=$x_3$](d);
\propag [plain] (b) to [edge label'=$x_5$](a);
\propag [plain] (a) to [out=-45, in=45, edge label=$x_2$](d);
\end{feynhand}
\end{tikzpicture}
\captionsetup{font=footnotesize}
\captionof{figure}{\texttt{debox}}
\label{fig:debox}
\end{minipage}

\subsection{Banana diagrams}

A {\it banana diagram} is a class of diagrams $\texttt{B}_n$ with two external edges and $n$ internal edges.
The diagrams $\texttt{B}_2$ and $\texttt{B}_3$ are referred to as {\it bubble diagram} and {\it sunrise diagram} respectively.
The Lee-Pomeransky polynomial is given by
$$
f_{\texttt{B}_{n}}(x;z)=\left( 1-\sum_{i=1}^nm_ix_i\right)\left( \sum_{i=1}^nx_1\cdots\widehat{x_i}\cdots x_n\right)+m_0x_1\cdots x_n,\ \ z=(m_0,m_1,\dots,m_n)\in\C^{n+1}.
$$
The variable $m_0$ is the {\it external mass} which is the scalar product $p^2$ and the variables $m_1,\dots,m_n$ are {\it internal masses}.

Let us consider the bubble diagram $\texttt{B}_2$.
The computation of $E^{\rm sub}$ can be done in a standard system of symbolic computation.
For example, a code snippet computing $E^{\rm sub}$ in \texttt{Oscar.jl} (version 1.3.1, \cite{OSCAR}) is the following:

\begin{minted}[breaklines,frame=lines]{julia}
using Oscar
R, v =polynomial_ring(QQ,[["x$i" for i=1:2];"y";["m$i" for i=0:2];["ξ$i" for i=0:2]])
x=v[1:2]
y=v[3]
m0,m1,m2=v[4:6]
ξ=v[7:end]
f=(1-sum([m1;m2].*x))*(x[1]+x[2])+m0*prod(x)
s=rand(Int64,1)[1]
ν=rand(Int64,2)
J0=[[ν[i]-s*y*x[i]*derivative(f,x[i]) for i=1:2]; [ξ[i]+s*y*derivative(f,v[i+3]) for i=1:3]; 1-y*f]
I0=eliminate(ideal(J0),[x;y])
o=degrevlex(R)
W=[[0 for i=1:6];[1 for i=1:3]]
oW = weight_ordering(W, o)
G=gens(groebner_basis(I0, ordering = oW))
nu0=tropical_semiring_map(QQ)
grI0=ideal([initial(g,nu0,-W) for g in G]);
PD=primary_decomposition(grI0)
Esub=[[eliminate(i[2],ξ);div(Oscar.degree(i[1]),Oscar.degree(i[2]))] for i in PD]
\end{minted}

The last output shows that $E^{\rm sub}$ is given by a formal sum
\begin{equation}\label{eq:Esub}
3[T^*_ZZ]+[T_{V(\Delta_2)}^*Z]+[T_{V_Z(m_0)}^*Z]+[T_{V(m_1)}^*Z]+[T_{V(m_2)}^*Z].    
\end{equation}
Here, we set $\Delta_2(m_0,m_1,m_2):=m_0^2+m_{1}^{2}+m_{2}^{2}-2m_{0}m_{1}-2\,m_{1}m_2-2m_{2}m_0$.
We claim that $E^{\rm hyp}_{\texttt{B}_2}$ is given by \eqref{eq:Esub}.
In fact, it is enough to check that the coefficient of each component correctly computes the drop of the Euler characteristic in view of the formula \eqref{eq:Euler index}.
Euler characteristic of a fiber is efficiently computed by the formula \cite[Theorem 1]{huh2013maximum} using, e.g., \texttt{HomotopyContinuation.jl} (\cite{breiding2018homotopycontinuation}).
Note that at a smooth point of a variety $Y\subset Z$, the value of the Euler obstruction ${\rm Eu}_Y$ is one.

Next,we consider the sunrise diagram $\texttt{B}_3$.
The hypergeometric discriminant is given by 
$$
E^{\rm hyp}_{\texttt{B}_3}=7[T^*_ZZ]+[T_{V(\Delta_3)}^*Z]+3\sum_{i=0}^3[T_{V(m_i)}^*Z]+\sum_{0\leq i<j\leq 3}[T_{V(m_i,m_j)}^*Z].
$$
Here, $\Delta_3(m_0,m_1,m_2,m_3)$ is a polynomial $e_{1}^{4}-8\,e_{1}^{2}e_{2}+16\,e_{2}^{2}-64\,e_{4}$, where $e_i$ denotes the $i$-th elementary symmetric polynomial in $m_0,m_1,m_2,m_3$.
The polynomial $\Delta_3(m_0,m_1,m_2,m_3)$ also appeared in \cite[Example 6.2]{telen2024euler}.
It is noteworthy that there are six components $T_{V(m_i,m_j)}^*Z$ whose projection to $Z$ have codimension two.
In view of Kashiwara's index formula \eqref{eq:Euler index}, this contributes to Euler characteristics.
For example, on a generic point $z$ of $V(m_1)$, the signed Euler characteristic is 
$\chi_z=7-3Eu_{V(m_1)}(z)=4$.
On the other hand, if $z$ is a generic point of $V_Z(m_1.m_2)$, the signed Euler characteristic is 
$$
\chi_z=7-3Eu_{V(m_1)}(z)-3Eu_{V(m_2)}(z)+Eu_{V(m_1,m_2)}(z)=7-3-3+1=2.
$$
The linearity of $Z$ naturally leads to a study of the projective dual of $V(\Delta_3)$.
In fact, $\Delta_3(m_0,m_1,m_2,m_3)$ is a homogeneous polynomial in $(m_0,m_1,m_2,m_3)$ and  $V(\Delta_3)$ can be regarded as a projective variety.
Let $Z^*$ be the dual vector space of $Z=\C^{4}$.
Then, there is a canonical identification $T^*Z=Z\times Z^*$.
We regard $Z^*$ as ${\rm Spec}\,\C[\mu_0,\mu_1,\mu_2,\mu_3]$ where $(\mu_0,\mu_1,\mu_2,\mu_3)$ is the dual coordinate of $(m_0,m_1,m_2,m_3)$.
The image of $T_{V(\Delta_3)}^*Z$ under the projection $T^*Z\to Z^*$ is the dual variety of $V(\Delta_3)$.
It turns out that the dual variety is a hypersurface defined as the Zariski closure of the set defined by a reciprocally linear equation
    $$
    \frac{1}{\mu_0}+\frac{1}{\mu_1}+\frac{1}{\mu_2}+\frac{1}{\mu_3}=0.
    $$
In particular, this gives rise to a rational parametrization of $V(\Delta_3)$.
The left surface in the following figure is the dehomogenized reciprocal hyperplane
$\{ (\mu_1,\mu_2,\mu_3)\in\mathbb{R}^3\mid 1+\frac{1}{\mu_1}+\frac{1}{\mu_2}+\frac{1}{\mu_3}=0\}$ and the right surface is the dehomogenized surface $\{(m_1,m_2,m_3)\in\mathbb{R}^3\mid \Delta(1,m_1,m_2,m_3)=0\}$.

\begin{figure}
    \centering
    \begin{minipage}{.45\textwidth}
    \includegraphics[width=\linewidth]{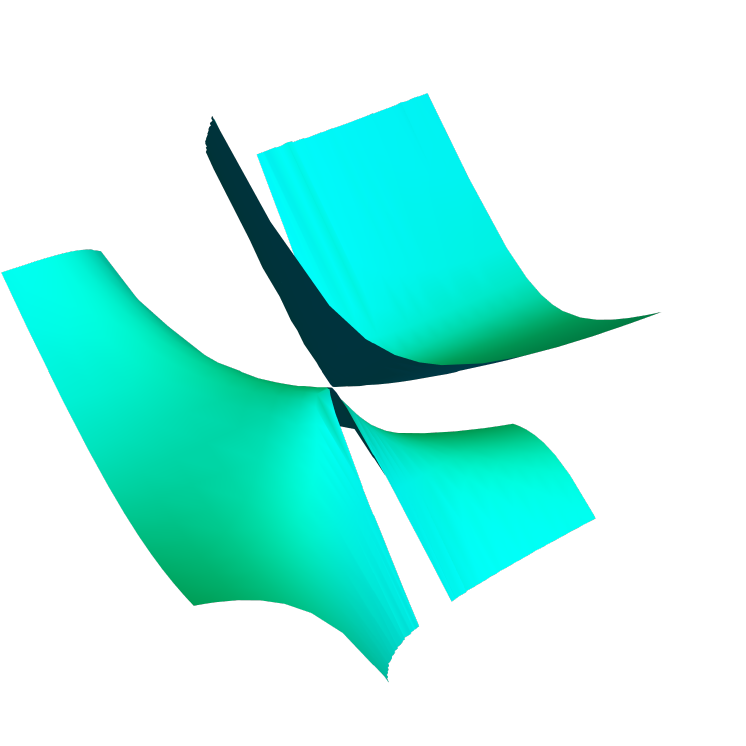}
    \caption{Reciprocal hyperplane}
    \label{fig:reciprocal}
    \end{minipage}
    \begin{minipage}{.45\textwidth}
    \includegraphics[width=\textwidth]{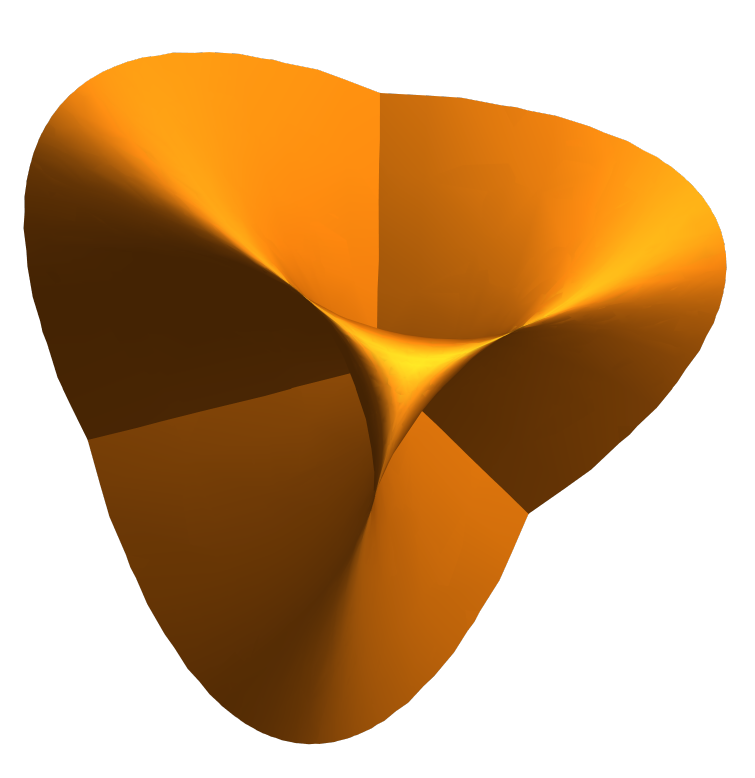}
    \caption{Dehomogenized $V(\Delta_3)$}
    \label{fig:dehomogenized}
    \end{minipage}
\end{figure}

As for $n=4$, the result is larger and harder to compute.
At least, we find that the Euler discriminant locus is $V(m_0\cdots m_4\Delta_4)$ where $\Delta_4=\Delta_4(m_0,\dots,m_4)$ is a symmetric polynomial of degree $8$.
By abuse of notation, let $e_i$ denotes the $i$-th elementary symmetric polynomial in $m_0,m_1,\dots,m_4$.
Then, we obtain
\begin{align*}
    \Delta_4(m_0,\dots,m_4)=&e_{1}^{8}-16\,e_{1}^{6}e_{2}+96\,e_{1}^{4}e_{2}^{2}-128\,e_{1}^{4}e_{4}-
      2\,048\,e_{1}^{3}e_{5}-256\,e_{1}^{2}e_{2}^{3}+1024\,e_{1}^{2}e_{2}e_{4
      }\\
    &+8192\,e_{1}e_{2}e_{5}+256\,e_{2}^{4}-2048\,e_{2}^{2}e_{4}-16\,384\,
      e_{3}e_{5}+4\,096\,e_{4}^{2}.
\end{align*}
$V(\Delta_4)$ is again projectively dual to a reciprocal hyperplane 
$$
\frac{1}{\mu_0}+\frac{1}{\mu_1}+\frac{1}{\mu_2}+\frac{1}{\mu_3}+\frac{1}{\mu_4}=0,
$$
that is, the vanishing locus of the elementary symmetric polynomial of degree four.
Hypergeometric discriminant also captures higher codimensional components.
Computing the signed Euler characteristics suggests that the following formula is true:
$$
E^{\rm hyp}_{\texttt{B}_4}=15[T^*_ZZ]+[T_{V(\Delta_4)}^*Z]+7\sum_{i=0}^4[T_{V(m_i)}^*Z]+3\sum_{0\leq i<j\leq 4}[T_{V(m_i,m_j)}^*Z]+\sum_{0\leq i<j<k\leq 4}[T_{V(m_i,m_j,m_k)}^*Z].
$$
Let $\mathbb{P}(Z)$ denote the quotient $(Z\setminus\{0\})//\C^*$ where $\C^*$ acts on $Z$ by scaling, regarding $Z$ as a $\C$-vector space.
We conjecture the following structure of $E^{\rm hyp}_{\texttt{B}_n}$:

\begin{conjecture}
The hypergeometric discriminant of $\texttt{B}_{n}$ is given by a formal sum
$$
E^{\rm hyp}_{\texttt{B}_{n}}=n_0[T^*_ZZ]+[T^*_{V(\Delta_{n})}Z]+\sum_{p=1}^{n-1}n_p\sum_{0\leq i_1<\cdots<i_p\leq n}[T^*_{V(m_{i_1},\dots,m_{i_p})}Z],
$$
where $n_p=2^{n-p}-1$.
The variety $V_{\mathbb{P}(Z)}(\Delta_{n})$ is the projective dual to the reciprocal hyperplane 
$$
\frac{1}{\mu_0}+\frac{1}{\mu_1}+\cdots+\frac{1}{\mu_n}=0.
$$
In particular, $V_{\mathbb{P}(Z)}(\Delta_{n})$ is a rational hypersurface.
A rational parametrization is given by $\mathbb{P}^{n-1}\ni [t_1:\dots:t_{n}]\mapsto \left[\left(\sum\limits_{i=1}^{n}t_i\right)^2:t_1^2:\dots:t_{n}^2\right]\in\mathbb{P}(Z)$.
Moreover, $\Delta_n$ is the Landau discriminant of the banana diagram $\texttt{B}_n$ as in \cite[Proposition 2]{mizera2022landau}.
\end{conjecture}
\noindent
Note that $s,m_i$ in \cite[Proposition 2]{mizera2022landau} corresponds to $m_0,\sqrt{m_i}$ in this paper.

\subsection{Parachute diagram and non-planar box}

In this section, we consider the {\it parachute diagram} and the {\it non-planar box} diagram.
The Lee-Pomeransky polynomial $f_{\texttt{par}}(x;z)$ of the diagram \texttt{par} is given by
\begin{align*}
\left(1-\sum_{i=1}^4m_ix_i\right)\left( x_1x_2+(x_1+x_2)(x_3+x_4)\right)+ x_1x_2(M_1x_3 + M_2x_4)+M_3x_3x_4(x_1+x_2).
\end{align*}
Here, $m_1,m_2,m_3,m_4$ represent {\it internal masses} while $M_i=p_i^2$ $(i=1,2,3)$ represent {\it external masses}.
To simplify the result, we assume $m_2=m_3=M_2=0$ so that $Z={\rm Spec}\,\C[m_1,m_4,M_1,M_3]$.
We list the data $(\pi_{T^*Z}(V(\mathfrak{p})),m_\mathfrak{p})$ which determines the hypergeometric discriminant $E^{\rm hyp}_{\texttt{par}}$.
Insterad of $\pi_{T^*Z}(V(\mathfrak{p}))$, we show its defining ideal $\mathfrak{p}_{\rm proj}$:

\begin{table}[h]
    \centering
    \begin{tabular}{c|c}
      ${\rm codim}V(\mathfrak{p}_{\rm proj})$   &$(\mathfrak{p}_{\rm proj},m_\mathfrak{p})$\\
      \hline
      0   & $(\left\{0\right\},8)$\\
      \hline
      1&$(\left\{m_{4}M_{1}-m_{1}M_{3}\right\},1),(\left\{m_{1}-m_{4}-M_{1}+M_{3}\right\},1),(\left\{m_{1}-m_{4}\right\},1)$,\\
       &$(\left\{m_{1}-M_{1}\right\},1),(\left\{m_{1}\right\},4),(\left\{m_{4}+M_{1}-M_{3}\right\},1),(\left\{m_{4}-M_{3}\right\},2)$,\\
       &$(\left\{m_{4}\right\},2),(\left\{M_{1}-M_{3}\right\},2),(\left\{M_{1}\right\},2),(\left\{M_{3}\right\},2)$\\
       \hline
      2&$(\left\{m_{4}+M_{1}-M_{3},m_{1}\right\},1),(\left\{m_{4}-M_{3},m_{1}-M_{1}\right\},1),(\left\{m_{4}-M_{3},m_{1}\right\},1)$,\\
       &$(\left\{m_{1},m_{4}\right\},2),(\left\{M_{1}-M_{3},m_{1}-m_{4}\right\},1),(\left\{M_{1}-M_{3},m_{1}\right\},1)$,\\
       &$(\left\{M_{1}-M_{3},m_{4}\right\},1),(\left\{m_{1},M_{1}\right\},2),(\left\{M_{1},m_{4}-M_{3}\right\},1),(\left\{m_{1},M_{3}\right\},1)$,\\
       &$(\left\{m_{4},M_{3}\right\},2),(\left\{M_{1},M_{3}\right\},2)$\\
\hline
      3&$(\left\{M_{1}-M_{3},m_{4},m_{1}\right\},1),(\left\{M_{1},m_{4}-M_{3},m_{1}\right\},1),(\left\{m_{1},m_{4},M_{3}\right\},1)$,\\
       &$(\left\{m_{1},M_{1},M_{3}\right\},1),(\left\{m_{4},M_{1},M_{3}\right\},1)$\\
       \hline
      4&$(\left\{m_{1},m_{4},M_{1},M_{3}\right\},1)$
    \end{tabular}
\end{table}
\noindent
Higher codimensional components of $E^{\rm hyp}_{\texttt{par}}$ appear to be tricky.
Many of the intersection of codimension one components appear while some are missing such as $\{m_1- m_4,M_1\}$.

As a bigger example of a Lee-Pomeransky polynomial, let us consider a non-planar box diagram \texttt{env}.
We employed \texttt{Groebner.mpl} (\cite{maple}) to compute this and the examples in the next sections.
To be more precise, we choose all the internal and external masses to be zero.
The remaining kinematic variables are {\it Mandelstam invariants} $s=(p_1+p_2)^2=2p_1\cdot p_2$ and $t=(p_1+p_4)^2=2p_1\cdot p_4$.
The Lee-Pomeransky polynomial $f_{\texttt{env}}(x;z)$ is given by 

\begin{align*}
   &(-s - t)x_1x_2x_3x_4 + x_1x_2x_3 + x_1x_2x_4 + x_1x_2x_5 + x_1x_3x_4 + tx_1x_3x_5x_6 + x_1x_3x_5\nonumber\\
   & + x_1x_3x_6+ x_1x_4x_6+ x_1x_5x_6 + x_2x_3x_4 + x_2x_3x_6 + sx_2x_4x_5x_6 + x_2x_4x_5 + x_2x_4x_6\nonumber\\
   &+ x_2x_5x_6 + x_3x_4x_5 + x_3x_5x_6 + x_4x_5x_6 .
\end{align*}
For this example, we did not succeed in computing the primary decomposition of ${\rm gr}(I_0)$.
However, it is still possible to compute the prime decomposition of its radical $\sqrt{{\rm gr}(I_0)}$ to find all the irreducible components $V(\mathfrak{p})$ of $E^{\rm hyp}$.
By combining this with Euler characteristic computation and Kashiwara's index theorem \eqref{eq:Euler index}, we obtain
\begin{equation*}
    E^{\rm hyp}_{\texttt{env}}=10[T^*_{\C^2}\C^2]+6[T^*_{V(s)}\C^2]+6[T^*_{V(t)}\C^2]+6[T^*_{V(s+t)}\C^2]+8[T^*_{V(s,t)}\C^2].
\end{equation*}

\subsection{Slashed box}
In this section, we consider the slashed box diagram \texttt{acn}.
Landau singularity of this diagram is studied in \cite{helmer2024landau} from the viewpoint of Whitney stratification.
We take the internal masses to be zero in this diagram.
Therefore, the Lee-Pomeransky polynomial is given by
\begin{align*}
    f_{\texttt{acn}}(x;z)=&sx_2x_4x_5 + tx_1x_3x_5 + M_1x_1x_2x_4 + M_1x_1x_3x_4 + M_1x_1x_4x_5 + M_2x_1x_2x_5 + M_3x_1x_2x_3 \nonumber\\
    &+ M_3x_2x_3x_4 + M_3x_2x_3x_5 + M_4x_3x_4x_5 + x_1x_2 + x_1x_3 + x_1x_5 + x_2x_4 + x_2x_5 \nonumber\\
    &+ x_3x_4 + x_3x_5 + x_4x_5.
\end{align*}
It is still expensive to compute the Euler discriminant locus.
Therefore, we consider specializations $M_i=M_j=0$ for various $i,j$.
Due to the symmetry of the diagram, it is classified into the following three cases.

\vspace{1em}

\noindent
\underline{Case  $M_1=M_3=0$:}
In this case, we can compute the hypergeometric discriminant
\begin{equation}\label{eq:M_1=M_3=0}
    E^{\rm hyp}_{\texttt{acn}}=2[T^*_{\C^2}\C^2]+[T^*_{V(s-M_4)}\C^2]+[T^*_{V(t-M_2)}\C^2]+[T^*_{V(s+t-M_2-M_4)}\C^2]+[T^*_{V(s-M_4,t-M_2)}\C^2].
\end{equation}
In particular, the Euler discriminant locus $\nabla_\chi(Z)$ is given by the vanishing locus of $(s-M_4)(t-M_2)(s+t-M_2-M_4)$, which is identical with the Landau singularity in \cite{helmer2024landau}.
As the Euler discriminant locus is always a subset of the birfurcation locus of $\pi :X\to Z$ and the birfurcation locus is contained in the union of a Whitney stratification considered in \cite{helmer2024landau}, it is proven that the bifurcation locus is given by the vanishing locus of \eqref{eq:M_1=M_3=0}.

\vspace{1em}

\noindent
\underline{Case $M_1=M_2=0$:}
In this case, we do not succeed in computing the hypergeometric discriminant.
The primary decomposition of ${\rm gr}(I_0)$ is currently out of reach.
On the other hand, elimination algorithm is easily applied and the Euler discriminant locus $\nabla_\chi(Z)$ is given by the vanishing locus of a polynomial
\begin{align}
    M_3M_4 s t\left({M_4} -t \right) \left({M_4} -s -t \right) \left({M_3} -t \right)\Delta_2(s,M_3,M_4)\left({M_4} {M_3} -{M_3} t -{M_4} t +s t +t^{2}\right).\label{eq:M1=M2=0}
\end{align}
The result is the same as that in \cite{helmer2024landau}.
Again, it is proved that the bifurcation locus is given by the vanishing locus of \eqref{eq:M1=M2=0}.

\vspace{1em}

\noindent
\underline{Case $M_2=M_4=0$:}
This case has the same difficulty as the case of $M_1=M_2=0$.
The Euler discriminant locus $\nabla_\chi(Z)$ is given by the vanishing locus of a polynomial
\begin{equation*}
\begin{split}
  {M_1} {M_3} s t \left({M_3} -t \right) \left({M_3} -s \right) \left({M_1} -t \right) \left({M_1} -s \right) \left({M_1} +{M_3} -s -t \right)\\
    \left(4 {M_1} {M_3} -s^{2}-2 s t -t^{2}\right) \left({M_1} {M_3} -s t \right).
\end{split}
\end{equation*}
At the moment, we cannot compare this result with the codimension one component of a Whitney stratification.
We only point out that the principal Landau determinant, a candidate for the Euler discriminant locus (\cite{fevola2024principal}), misses one component.

\subsection{Rationality of the singularity}
Let us conclude this section with yet another example of a Feynman diagram called \texttt{debox}.
By choosing internal masses to be all zero, the Lee-Pomeransky polynomial $f_{\texttt{debox}}(x;z)$ is
\begin{equation*}
\begin{split}
&M_1 x_1 x_2 x_5 + M_2 x_1 x_2 x_3 + M_3 x_1 x_3 x_4 + M_3 x_2 x_3 x_4 + M_4 x_1 x_4 x_5 + M_4 x_2 x_4 x_5  \\
&+ sx_1 x_3 x_5 + sx_2 x_3 x_5 + tx_1 x_2 x_4  + x_1 x_2 + x_1 x_3 + x_1 x_4 + x_1 x_5 + x_2 x_3 + x_2 x_4 + x_2 x_5 .
\end{split}
\end{equation*}

\noindent
Let us set $M_4=0$.
The Euler discriminant locus is given by the vanishing locus of the following polynomial:
\begin{equation*}
    M_1 M_2 M_3 s  (M_3 - s) (M_1 - t) (M_1 M_3 - s t) \Delta_2(s,M_1,M_2) \Delta_2(t,M_2,M_3)\lambda(s,t,M_1,M_2,M_3).
\end{equation*}
Here, $\lambda(s,t,M_1,M_2,M_3)$ is a cubic polynomial given by
$$
M_1^2  M_3 - M_1 M_2 M_3 + M_1 M_2 s + M_1 M_3^2  - M_1 M_3 s - M_1 M_3 t - M_1 s t + M_2 M_3 t - M_2 s t - M_3 s t + s^2  t + s t^2.
$$
Since the polynomial is linear in $M_2$, its vanishing locus is manifestly rational.
Namely, the projective variety defined as the vanishing locus of $\lambda(s,t,M_1,M_2,M_3)$ is singular in view of the irrationality of a smooth cubic threefold by Clemens and Griffith (\cite{clemens1972intermediate}).
In fact, the projective singular locus has dimension two.
As far as we know, examples of Euler discriminant loci associated to Lee-Pomeransky polynomials are always rational.
It is interesting to study the rationality and a rational parametrization of an irreducible component of an Euler discriminant locus.
The case of principal $A$-determinants is well-known: each irreducible component is an $A$-discriminant locus of which a rational parametrization is known as {\it Horn-Kapranov uniformization} (\cite{HuhMLdegreeOne},\cite{kapranov1991characterization}).
We remark that one can produce an artificial example of a non-rational Euler discriminant locus.
A simple example is $f(x;z)=z_1^2(x_1+x_2+x_3)x_3+(z_2^3+1)x_1x_2$ whose Euler discriminant locus is $z_1(z_2^3+1)(z_1^2-z_2^3-1)=0$.

\section*{Acknowledgements}
The author express his sincere gratitude for Dan Bath, Claudia Fevola, Anna-Laura Sattelberger, Bernd Sturmfels, Mahrud Sairafi, Avi Steiner and Nobuki Takayama for discussions and comments.
He thanks Felix Tellander and Martin Helmer for a discussion on Whitney stratification.
Finally, he thanks Simon Telen and Maximilian Weismann for a discussion on their work \cite{telen2024euler}, which lead to the results in \S\ref{subsec:geometric description}.
This research is supported by JSPS KAKENHI Grant Number 22K13930.

\bibliographystyle{abbrv}
\bibliography{references.bib}

\vfill 

\noindent{\bf Author's address:}
\smallskip

\noindent Graduate School of Information Sciences, Division of Mathematics, Tohoku University

\noindent
{\tt saiei@tohoku.ac.jp}
\end{document}